\documentclass[11pt]{article}

\usepackage{etex}
\reserveinserts{28}
\usepackage[all]{xy}

\usepackage{amsfonts}
\usepackage{amsthm}
\usepackage{enumerate}
\usepackage{graphicx}
\usepackage{mathrsfs}
\usepackage{bm}
\usepackage{cite}
\usepackage[colorlinks,linkcolor=blue,citecolor=blue]{hyperref}
\usepackage{amssymb,amsmath} % font used for R in Real numbers
\usepackage{hyperref}
\usepackage{makecell}
\usepackage{pgf}
\usepackage{pgffor}
\usepackage{pgfcalendar}
\usepackage{pgfpages}
\usepackage{shuffle,yfonts}
\usepackage{mathtools}
\usepackage{tikz}
\usetikzlibrary{arrows,shapes,chains,patterns}
 \allowdisplaybreaks

\DeclareFontFamily{U}{shuffle}{}
\DeclareFontShape{U}{shuffle}{m}{n}{ <-8>shuffle7 <8->shuffle10}{}
\newcommand{\sha}{\shuffle}

\newcommand{\ES}{\mathsf {ES}}
\newcommand{\MZV}{\mathsf {MZV}}
\newcommand{\MtV}{\mathsf {MtV}}
\newcommand{\MTV}{\mathsf {MTV}}
\newcommand{\MSV}{\mathsf {MSV}}
\newcommand{\MMV}{\mathsf {MMV}}
\newcommand{\MMVo}{\mathsf {MMVo}}
\newcommand{\MMVe}{\mathsf {MMVe}}

\newcommand{\gd}{\delta}
\newcommand{\lra}{\longrightarrow}
\newcommand{\Lra}{\Longrightarrow}
\newcommand\Res{{\rm Res}}
\newcommand\setX{{\mathsf{X}}}
\newcommand\fA{{\mathfrak{A}}}
\newcommand\evaM{{\texttt{M}}}
\newcommand\evaML{{\text{\em{\texttt{M}}}}}
\newcommand\z{{\texttt{z}}}
\newcommand\tx{{\texttt{x}}}
\newcommand\txp{{\tx_1}} % textstyle x positive 1
\newcommand\txn{{\tx_{-1}}} % textstyle x negative 1

\newcommand\om{{\omega}}
\newcommand\omn{\omega_{-1}}
\newcommand\omz{\omega_0}
\newcommand\omp{\omega_{1}}
\newcommand\eps{{\varepsilon}}
\newcommand\bfeps{{\boldsymbol \eps}}
\newcommand\bfs{{\bf s}}
\newcommand\bfn{{\bf n}}
\newcommand\bfw{{\bf w}}
\newcommand\bfj{{\bf j}}
\newcommand\bfk{{\bf k}}
\newcommand\bfl{{\bf l}}
\newcommand\bfq{{\bf q}}
\newcommand\bfu{{\bf u}}
\newcommand\bfv{{\bf v}}

\newcommand{\myone}{{1}}

\catcode`!=11
\let\!int\int \def\int{\displaystyle\!int}
\let\!lim\lim \def\lim{\displaystyle\!lim}
\let\!sum\sum \def\sum{\displaystyle\!sum}
\let\!sup\sup \def\sup{\displaystyle\!sup}
\let\!inf\inf \def\inf{\displaystyle\!inf}
\let\!cap\cap \def\cap{\displaystyle\!cap}
\let\!max\max \def\max{\displaystyle\!max}
\let\!min\min \def\min{\displaystyle\!min}
\let\!frac\frac \def\frac{\displaystyle\!frac}
\catcode`!=12

\let\oldsection\section
\renewcommand\section{\setcounter{equation}{0}\oldsection}

\allowdisplaybreaks

\DeclareMathOperator*{\dep}{dep}

\newcommand\xx{\mbox{\bfseries \itshape x}}
\def\R{\mathbb{R}}
\def\N{\mathbb{N}}\def\Z{\mathbb{Z}}
\def\Q{\mathbb{Q}}

\def\su{\sum\limits_{n=1}^\infty}

\def\t{\widetilde{t}}
\def\S{\widetilde{S}}

\def\tt{\left(\frac{1-t}{1+t} \right)}
\def\xx{\left(\frac{1-x}{1+x} \right)}

\def\ol{\overline}

\theoremstyle{plain}
\newtheorem{thm}{Theorem}[section]
\newtheorem{lem}[thm]{Lemma}
\newtheorem{cor}[thm]{Corollary}
\newtheorem{con}[thm]{Conjecture}
\newtheorem{pro}[thm]{Proposition}
\theoremstyle{definition}
\newtheorem{defn}{Definition}[section]
\newtheorem{re}[thm]{Remark}
\newtheorem{exa}[thm]{Example}

\begin{document}
%%%%%%%%%%%%%%%%%%%% title %%%%%%%%%%%%%%%%%%%%%%%%%%%%%%%%%%%%%%%%%%%%%%%%
\title{\bf Variants of Multiple Zeta Values with Even and Odd Summation Indices}
%{\bf Evaluations For Several $T$-Variants Of Multiple Zeta Values Of Kaneko-Yamamoto Type}

%submitted to Math. Zeit., 07/12/2020

\author{
{Ce Xu${}^{a,}$\thanks{Email: cexu2020@ahnu.edu.cn}\quad and Jianqiang Zhao${}^{b,}$\thanks{Email: zhaoj@ihes.fr}}\\[1mm]
\small a. School of Math. and Statistics, Anhui Normal University, Wuhu 241000, P.R. China\\
\small b. Department of Mathematics, The Bishop's School, La Jolla, CA 92037, USA\\
[5mm]
Dedicated to professor Masanobu Kaneko on the occasion of his 60th birthday}

\date{}
\maketitle

\noindent{\bf Abstract.} In this paper, we define and study a variant of multiple zeta values of level 2 (which is called multiple mixed values or multiple $M$-values, MMVs for short), which forms a subspace of the space of alternating multiple zeta values. This variant includes both Hoffman's multiple $t$-values and Kaneko-Tsumura's multiple $T$-values as special cases.
We set up the algebra framework for the double shuffle relations (DBSFs) of the MMVs, and exhibits nice properties such as duality, integral shuffle relation, series stuffle relation, etc., similar to ordinary multiple zeta values.
Moreover, we study several $T$-variants of Kaneko-Yamamoto type multiple zeta values by establishing some explicit relations between these $T$-variants and Kaneko-Tsumura $\psi$-values. Furthermore, we prove that all Kaneko-Tsumura $\psi$-values can be expressed in terms of Kaneko-Tsumura multiple $T$-values by using multiple associated integrals, and find some duality formulas for Kaneko-Tsumura $\psi$-values. We also discuss the explicit evaluations for a kind of MMVs of depth two and three by using the method of contour integral and residue theorem. Finally, we investigate the dimensions of a few interesting subspaces of MMVs for small weights.


%\medskip
%\noindent{\bf Declarations.}

%\textit{Funding.} Not applicable

%\textit{Conflicts of interest/Competing interests.} Not applicable

%\textit{Availability of data and material.} Not applicable

%\textit{Code availability.} MAPLE

\medskip

%\tableofcontents

%\setcounter{section}{-1}
\section{Introduction and Notations}
Between late 1742 and early 1743, Euler first touched on the subject of the double zeta star values in a series of correspondence with Goldbach. In modern notation, these are defined as follows:
\begin{equation*}
\zeta^*(r,s)=\sum_{m=1}^\infty \sum_{n\ge m} \frac{1}{m^r n^s}.
\end{equation*}
Euler returned to the same subject after about 30 years and discovered the now famous Euler's decomposition formula in \cite{Euler1776}. More than two hundred years later, these objects were generalized to the so-called multiple zeta values (MZVs) independently at almost the same time by Hoffman \cite{H1992} and Zagier \cite{DZ1994}:
\begin{equation}\label{equ:MZVdefn}
 \zeta(k_1,\ldots,k_r):=\sum\limits_{0<m_1<\cdots<m_r} \frac{1}{m_1^{k_1}\cdots m_r^{k_r}}
\end{equation}
for any positive integers $k_1,\ldots,k_r$ such that $k_2\ge 2$.
The primary motivation of this paper is the study a number of variations of the MZVs when we allow only even or odd indices $k_1,\ldots,k_r$ in \eqref{equ:MZVdefn}.

We begin with some basic notations. A finite sequence $\bfk=\bfk_r:= (k_1,\ldots, k_r)$ of positive integers is called a \emph{composition}. As usual, we put
\[|\bfk|:=k_1+\cdots+k_r,\quad \dep(\bfk):=r,\]
and call them the weight and the depth of $\bfk$, respectively. If $k_r>1$, $\bfk$ is called \emph{admissible}. For ${\bfk_r}:= (k_1,\ldots, k_r)$, set ${\bfk_0}:=\emptyset$ and $({\bfk_r})_{+}:=(k_1,\ldots,k_{r-1},k_r+1)$.

\subsection{Multiple harmonic sums and multiple zeta values}
For a composition $\bfk=(k_1,\ldots,k_r)$ and positive integer $n$, the multiple harmonic sums (MHSs) and multiple harmonic star sums (MHSSs) are defined by
\begin{align*}
\zeta_n(k_1,k_2,\ldots,k_r):=\sum\limits_{0<m_1<\cdots<m_r\leq n } \frac{1}{m_1^{k_1}m_2^{k_2}\cdots m_r^{k_r}}
\end{align*}
and
\begin{align*}
\zeta^\star_n(k_1,k_2,\ldots,k_r):=\sum\limits_{0<m_1\leq \cdots\leq m_r\leq n} \frac{1}{m_1^{k_1}m_2^{k_2}\cdots m_r^{k_r}},
\end{align*}
respectively. If $n<k$ then ${\zeta_n}(\bfk):=0$ and ${\zeta _n}(\emptyset )={\zeta^\star _n}(\emptyset ):=1$.

For an admissible composition $\bfk$, the multiple zeta values (abbr. MZVs) and the multiple zeta-star value (abbr. MZSVs) are defined by
\begin{equation*}
\zeta(\bfk):=\lim_{n\to\infty } \zeta_n(\bfk)
\qquad\text{and}\qquad
\zeta^\star(\bfk):=\lim_{n\to\infty } \zeta_n^\star(\bfk),
\end{equation*}
respectively. The systematic study of multiple zeta values began in the early 1990s with the works of Hoffman \cite{H1992} and Zagier \cite{DZ1994}. Due to their surprising and sometimes mysterious appearance in the study of many branches of mathematics and theoretical physics, these special values have attracted a lot of attention and interest in past three decades (for example, see the book by the second author \cite{Z2016}).

Recently, Kaneko and Yamamoto \cite{KY2018} introduced and studied a new kind of multiple zeta values.
For any two compositions of positive integers $\bfk=(k_1,\ldots,k_r)$ and $\bfl=(l_1,\ldots,l_s)$, define
\begin{align}
\zeta(\bfk\circledast{\bf l^\star})&:=\sum\limits_{0<m_1<\cdots<m_r=n_s\geq \cdots \geq n_1 > 0} \frac{1}{m_1^{k_1}\cdots m_r^{k_r}n_1^{l_1}\cdots n_s^{l_s}}\nonumber\\
&=\su \frac{\zeta_{n-1}(k_1,\ldots,k_{r-1})\zeta^\star_n(l_1,\ldots,l_{s-1})}{n^{k_r+l_s}}.
\end{align}
We call them \emph{Kaneko-Yamamoto multiple zeta values} (K-Y MZVs for short). It is clear that the left-hand side is a certain integral which can be written as a $\Z$-linear combination of MZVs. The same holds if we use the series expression on the right. Kaneko and Yamamoto presented a new ``integral=series" type identity of multiple zeta values, and conjectured that this identity is enough to describe all linear relations of multiple zeta values over $\mathbb{Q}$. In \cite{KY2018}, ``$\bfk\circledast\bfl$" is defined by the `circle harmonic shuffle product' of compositions $\bfk=(k_1,\ldots,k_r)$ and $\bfl=(l_1,\ldots,l_s)$.

\subsection{Multiple mixed values}
Several other variants of MZVs have been introduced and studied.
In recent papers \cite{KTA2018,KTA2019}, Kaneko and Tsumura introduced and studied a new kind of
multiple zeta values of level two, called \emph{multiple T-values} (MTVs), defined by
\begin{align}
T(k_1,k_2,\ldots,k_r):&=2^r \sum_{0<m_1<\cdots<m_r\atop m_i\equiv i\ {\rm mod}\ 2} \frac{1}{m_1^{k_1}m_2^{k_2}\cdots m_r^{k_r}}\nonumber\\
&=2^r\sum\limits_{0<n_1<\cdots<n_r} \frac{1}{(2n_1-1)^{k_1}(2n_2-2)^{k_2}\cdots (2n_r-r)^{k_r}}.
\end{align}
This is in contrast to
Hoffman's \emph{multiple $t$-value} (MtVs) defined in \cite{H2019} as follows:
\begin{align}
t(k_1,k_2,\ldots,k_r):&=\sum_{0<m_1<\cdots<m_r\atop \forall m_i: odd} \frac{1}{m_1^{k_1}m_2^{k_2}\cdots m_r^{k_r}}\nonumber\\
&=\sum\limits_{0<n_1<\cdots<n_r} \frac{1}{(2n_1-1)^{k_1}(2n_2-1)^{k_2}\cdots (2n_r-1)^{k_r}}.
\end{align}
It is obvious that MtVs satisfy the series stuffle relation, however, it is nontrivial to see that MTVs can be expressed using iterated integral and satisfy both the duality relations (see \cite[Thm. 3.1]{KTA2019}) and the integral shuffle relations (see \cite[Thm. 2.1]{KTA2019}).

All the above variants of MZVs are level two special cases of the objects considered by Yuan and the second author in \cite[(2.1)]{YuanZh2014a} where they define more generally the MZVs of level $N$. To reformulate their definition at level two, for any admissible composition of positive integers $\bfk=(k_1,k_2,\ldots,k_r)$
and $\bfeps=(\eps_1, \dots, \eps_r)\in\{\pm 1\}^r$ we define the \emph{multiple mixed values} (or \emph{multiple $M$-values}, MMVs for short)
\begin{equation}\label{equ:MMVdefn}
M(\bfk;\bfeps):=\sum_{0<m_1<\cdots<m_r} \frac{(1+\eps_1(-1)^{m_1}) \cdots (1+\eps_r(-1)^{m_r})}{m_1^{k_1} \cdots m_r^{k_r}}=\sum_{\substack{0<n_1<\cdots<n_r\\ 2| n_j \text{ if } \eps_j=1 \\ 2\nmid n_j \text{ if } \eps_j=-1
}} \frac{2^r}{n_1^{k_1} \cdots n_r^{k_r}}.
\end{equation}
As usual, we call $k_1+\cdots+k_r$ and $r$ the \emph{weight} and \emph{depth}, respectively.
For convenience, we say the \emph{signature} of $k_j$ is \emph{even} or \emph{odd} depending on whether $\eps_j$ is 1 or $-1$.
It is again apparent that MMVs satisfy the series stuffle relations.
Moreover, MMVs can be expressed using iterated integrals which lead to the integral shuffle relations. Similar to MZVs, these can be extended to regularized shuffle relations using regularized non-admissible values (essentially some polynomials of $T$, see Theorem \ref{thm:RDSoverR}).

Note that MTVs are MMVs with alternating signatures starting with the odd signature (i.e., smallest summation index
in \eqref{equ:MMVdefn} is odd). To extend the duality relation to MMVs, we have to restrict to a subset. We say a MMV has \emph{odd} (resp.~\emph{even}) \emph{signature} (MMVo, resp.~MMVe, for short) if its smallest summation index in \eqref{equ:MMVdefn} is odd (resp.~even). See Theorem~\ref{thm:dualMMVo} for details.
On the other hand, a class of MMVs that is opposite to MTVs, called \emph{multiple S-values} (MSVs),
can be defined as follows. For any admissible composition $\bfk=(k_1,k_2,\ldots,k_r)$,
\begin{align}
S(k_1,k_2,\ldots,k_r):&=2^r \sum_{0<m_1<\cdots<m_r\atop m_i\equiv i-1\ {\rm mod}\ 2} \frac{1}{m_1^{k_1}m_2^{k_2}\cdots m_r^{k_r}}\nonumber\\
&=2^r \sum_{0<n_1<\cdots<n_r} \frac{1}{(2n_1)^{k_1}(2n_2-1)^{k_2}\cdots (2n_r-r+1)^{k_r}}.
\end{align}
Namely, MSVs are MMVs with alternating signatures starting with the even signature.
It is clear that every MMV can be written as a linear combination of alternating multiple zeta values
(also referred to as Euler sums or colored multiple zeta values) defined as follows.
For $\bfk\in\N^r$ and $\bfeps\in\{\pm\}^r$, if $(\bfk_r,\eps_r)\ne(1,1)$ (called \emph{admissible} case) then
\begin{equation*}
 \zeta(\bfk;\bfeps):=\sum\limits_{0<m_1<\cdots<m_r}\prod\limits_{j=1}^r \eps_j^{m_j} m_j^{-k_j}.
\end{equation*}
We may compactly indicate the presence of an alternating sign as follows.
Whenever $\eps_j=-1$,  we place a bar over the corresponding integer exponent $k_j$. For example,
\begin{equation*}
\zeta(\bar 2,3,\bar 1,4)=\zeta(  2,3,1,4;-1,1,-1,1).
\end{equation*}

\subsection{Variants of multiple harmonic sums}
For positive integers $m$ and $n$ such that $n\ge m$, we define
\begin{align*}
&D_{n,m} :=
\left\{
  \begin{array}{ll}
\Big\{(n_1,n_2,\dotsc,n_m)\in\N^{m} \mid 0<n_1\leq n_2< n_3\leq \cdots \leq n_{m-1}<n_{m}\leq n \Big\},\phantom{\frac12}\ & \hbox{if $2\nmid m$;} \\
\Big\{(n_1,n_2,\dotsc,n_m)\in\N^{m} \mid 0<n_1\leq n_2< n_3\leq \cdots <n_{m-1}\leq n_{m}<n \Big\},\phantom{\frac12}\ & \hbox{if $2\mid m$,}
  \end{array}
\right.  \\
&E_{n,m} :=
\left\{
  \begin{array}{ll}
\Big\{(n_1,n_2,\dotsc,n_{m})\in\N^{m}\mid 1\leq n_1<n_2\leq n_3< \cdots< n_{m-1}\leq n_{m}< n \Big\},\phantom{\frac12}\ & \hbox{if $2\nmid m$;} \\
\Big\{(n_1,n_2,\dotsc,n_{m})\in\N^{m}\mid 1\leq n_1<n_2\leq n_3< \cdots \leq n_{m-1}< n_{m}\leq n \Big\}, \phantom{\frac12}\ & \hbox{if $2\mid m$.}
  \end{array}
\right.
\end{align*}

\begin{defn} For positive integer $m$, define
\begin{align}
&T_n({\bfk_{2m-1}}):= \sum_{\bfn\in D_{n,2m-1}} \frac{2^{2m-1}}{(\prod_{j=1}^{m-1} (2n_{2j-1}-1)^{k_{2j-1}}(2n_{2j})^{k_{2j}})(2n_{2m-1}-1)^{k_{2m-1}}},\label{MOT}\\
&T_n({\bfk_{2m}}):= \sum_{\bfn\in D_{n,2m}} \frac{2^{2m}}{\prod_{j=1}^{m} (2n_{2j-1}-1)^{k_{2j-1}}(2n_{2j})^{k_{2j}}},\label{MET}\\
&S_n({\bfk_{2m-1}}):= \sum_{\bfn\in E_{n,2m-1}} \frac{2^{2m-1}}{(\prod_{j=1}^{m-1} (2n_{2j-1})^{k_{2j-1}}(2n_{2j}-1)^{k_{2j}})(2n_{2m-1})^{k_{2m-1}}},\label{MOS}\\
&S_n({\bfk_{2m}}):= \sum_{\bfn\in E_{n,2m}} \frac{2^{2m}}{\prod_{j=1}^{m} (2n_{2j-1})^{k_{2j-1}}(2n_{2j}-1)^{k_{2j}}},\label{MES}
\end{align}
where $T_n({\bfk_{2m-1}}):=0$ if $n<m$, and $T_n({\bfk_{2m}})=S_n({\bfk_{2m-1}})=S_n({\bfk_{2m}}):=0$ if $n\leq m$. Moreover, for convenience sake, we set $T_n(\emptyset)=S_n(\emptyset):=1$. We call \eqref{MOT} and \eqref{MET} \emph{multiple $T$-harmonic sums} ({\rm MTHSs} for short), and call \eqref{MOS} and \eqref{MES}  \emph{multiple $S$-harmonic sums} ({\rm MSHSs} for short).
\end{defn}

According to the definitions of MTHSs and MSHSs, we have the following relations
\begin{alignat*}{3}
&T_n({\bfk_{2m}})=2\sum_{j=1}^{n-1} \frac{T_j({\bfk_{2m-1}})}{(2j)^{k_{2m}}}, \qquad &
&T_n({\bfk_{2m-1}})=2\sum_{j=1}^{n} \frac{T_j({\bfk_{2m-2}})}{(2j-1)^{k_{2m-1}}},\\
&S_n({\bfk_{2m}})=2\sum_{j=1}^{n} \frac{S_j({\bfk_{2m-1}})}{(2j-1)^{k_{2m}}}, \qquad &
&S_n({\bfk_{2m-1}})=2\sum_{j=1}^{n-1} \frac{S_j({\bfk_{2m-2}})}{(2j)^{k_{2m-1}}}.
\end{alignat*}
It is clear that when taking the limit $n\rightarrow \infty$ in \eqref{MOT} and \eqref{MET}, \eqref{MOS} and \eqref{MES} with $k_r>1$, we get the MTVs and MSVs, respectively:
\begin{equation*}
T(\bfk)=\lim_{n\to\infty }T_n(\bfk),\qquad
S(\bfk)=\lim_{n\to\infty }S_n(\bfk).
\end{equation*}

Now, we use the MTHSs and MSHSs to define the convoluted $T$-values $T({\bfk_{r}}\circledast {\bfl_{s}})$, which can be regarded as a $T$-variant of K-Y MZVs.
\begin{defn} For positive integers $m$ and $p$, the \emph{convoluted $T$-values}
\begin{align}
&T({\bfk_{2m}}\circledast{\bfl_{2p}})=2\su \frac{T_n({\bfk_{2m-1}})T_n({\bfl_{2p-1}})}{(2n)^{k_{2m}+l_{2p}}},\\
&T({\bfk_{2m-1}}\circledast{\bfl_{2p-1}})=2\su \frac{T_n({\bfk_{2m-2}})T_n({\bfl_{2p-2}})}{(2n-1)^{k_{2m-1}+l_{2p-1}}},\\
&T({\bfk_{2m}}\circledast{\bfl_{2p-1}})=2\su \frac{T_n({\bfk_{2m-1}})S_n({\bfl_{2p-2}})}{(2n)^{k_{2m}+l_{2p-1}}},\\
&T({\bfk_{2m-1}}\circledast{\bfl_{2p}})=2\su \frac{T_n({\bfk_{2m-2}})S_n({\bfl_{2p-1}})}{(2n-1)^{k_{2m-1}+l_{2p}}}.
\end{align}
\end{defn}

\begin{re} (i). Here the `$\circledast$' is only a notation which doesn't satisfy the ``circle harmonic shuffle product'' relation. (ii) It is possible to defined convoluted $S$-values $S({\bfk_{r}}\circledast {\bfl_{s}})$ in a similar vein so that the first factor in the sum is a MSHS and the second factor is MSHS or MTHS depending whether $r$ and $s$ have the same parity or not.
\end{re}

Note that the MTVs are special cases of the convoluted $T$-values since
\begin{align*}
&T({\bfk_{r}}\circledast (1))=T((\bfk_r)_{+}),\quad T((1)\circledast {\bfl_{2p-1}})=T((\bfl_{2p-1})_{+}).
\end{align*}
Moreover, from the definition of $T({\bfk_{2m-1}}\circledast{\bfl_{2p}})$, we have
\begin{align*}
&T((1)\circledast \bfl_{2p})=S((\bfl_{2p})_{+}).
\end{align*}

The primary goals of this paper are to study the explicit relations of MMVs and its special type, and
establish some explicit evaluations of the convoluted $T$-values and related values via MTVs and (alternating) single zeta values.

The remainder of this paper is organized as follows. In Section \ref{sec:MMV} we find the series stuffle relations and integral shuffle relations of MMVs, and set up the algebra framework for the regularized double shuffle relations (DBSFs) of MMVs.

In Section \ref{sec:KTV}, we first prove four integral identities involving the nature logarithm function. Then we apply these formulas obtained to establish four explicit relations between $T({\bfk_{r}}\circledast {\bfl_{s}})$ and Kaneko-Tsumura $\psi$-function (see \eqref{a14}).

In Section \ref{sec:poset}, we review the definitions and basic properties of 2-labeled posets and the associated integrals introduced by Yamamoto \cite{Y2014}. Moreover, we give the integral expressions of multiple $T$-values and Kaneko-Tsumura $\psi$-values. Further, we apply the integral expressions of $\psi(k_1,\ldots,k_r;p+1)$ to prove that the $\psi(k_1,\ldots,k_r;p+1)$ can be expressed in terms of MTVs, and give explicit formula. Further, we prove that the $T(\bfk\circledast\{1\}_p)$ can be expressed in terms of products of {\rm MTVs} and alternating single zeta values, where for any string $\bfs$ we denote by $\{\bfs\}_p$ the string obtained by repeating $\bfs$ $p$-times.

In Section \ref{sec:psi-V}, we give explicit expressions of convoluted $T$-values by Kaneko-Tsumura $\psi$-values and alternating zeta values, and find some duality relations for Kaneko-Tsumura $\psi$-Values.

In Section \ref{sec:MSV}, we use the method of contour integration and residue theorem to evaluate the multiple $S$-values of depth two and three. Furthermore, we also prove an reducible theorem for MMVs of depth three.

In Section \ref{sec:dimMMV}, we find out the dimensions of MMVs, MtVs, MTVs and MSVs, and give some conjectures on relations between MMVs, MtVs, MTVs and MSVs.

CX expresses his deep gratitude to Prof. Masanobu Kaneko and Prof. Weiping Wang for valuable discussions and comments. JZ wants to thank Prof. Kaneko for inviting him to visit the Multiple Zeta Research Center at Kyushu University where this joint work started.

\section{Multiple Mixed Values}\label{sec:MMV}
Recall that MZVs are equipped with regularized double shuffle relations (DBSFs) which are generally believed to generate all $\Q$-linear relations among MZVs. Furthermore, the regularized DBSFs can be defined for all colored MZVs although it is known that they are not enough to generate all $\Q$-linear  relations for higher levels. For example, Euler sums (i.e., level two values) satisfy the so-called distribution relations which are not always contained in the regularized DBSFs (see \cite{BlumleinBrVe2010,Zhao2010a}).

In this section, we will consider the similar structures as above among multiple mixed values defined as follows.
For any admissible composition of positive integers $\bfk=(k_1,k_2,\ldots,k_r)$ and $\eps_1, \dots, \eps_r\in\{\pm 1\}$, the \emph{multiple mixed values} (or \emph{multiple $M$-values}, MMVs for short) are defined by
\begin{equation*}
M(\bfk;\bfeps):=\sum_{0<m_1<\cdots<m_r} \frac{(1+\eps_1(-1)^{m_1}) \cdots (1+\eps_r(-1)^{m_r})}{m_1^{k_1} \cdots m_r^{k_r}}=\sum_{\substack{0<n_1<\cdots<n_r\\ 2| n_j \text{ if } \eps_j=1 \\ 2\nmid n_j \text{ if } \eps_j=-1
}} \frac{2^r}{n_1^{k_1} \cdots n_r^{k_r}}.
\end{equation*}
As usual, we call $k_1+\cdots+k_r$ and $r$ the \emph{weight} and \emph{depth}, respectively.
For convenience, we say the \emph{signature} of $k_j$ is even or odd depending on whether $\eps_j$ is 1 or $-1$.
For brevity, we put a check on top of the component $k_j$ if $\eps_j=-1$. For example,
\begin{align*}
M(1,2,\check{3})=&\, \sum_{0<m_1<m_2<m_3} \frac{(1+(-1)^{m_1}) (1+(-1)^{m_2}) (1-(-1)^{m_3})}{m_1 m_2^{2}m_3^{3}}\\
=&\, \sum_{0<\ell<m<n} \frac{8}{(2\ell)  (2m)^{2} (2n-1)^{3}}.
\end{align*}
It is again apparent that MMVs satisfy the series stuffle relations.  For example,
\begin{align*}
M(2,1,\check{3})M(\check{2})=&\,M(\check{2},2,1,\check{3})+M(2,\check{2},1,\check{3})
+M(2,1,\check{2},\check{3})+M(2,1,\check{3},\check{2})+2M(2,1,\check{5}),\\
M(1,\check{3})M(2,\check{3})=&\,2M(1,2,\check{3},\check{3})+2M(2,1,\check{3},\check{3})
+M(2,\check{3},1,\check{3})+M(1,\check{3},2,\check{3})\\
&\, +
4M(3,\check{3},\check{3})+2M(1,2,\check{6})+2M(2,1,\check{6})+4M(3,\check{6}).
\end{align*}
We observe that stuffing can happen only when the two merging components have the same signature, and
an extra factor of 2 must appear for each incidence of stuffing.

To study the shuffle structure among MMVs, we first recall the set-up for Euler sums. Define the 1-forms
\begin{equation*}
\tx_0(t)=\frac{dt}{t}, \quad \txp(t)=\frac{dt}{1-t}, \quad \txn(t)=\frac{-dt}{1+t}.
\end{equation*}
Then
\begin{equation*}
\begin{split}
\zeta(k_1,\dots,k_r;\eps_1,\eps_2,\dots,\eps_r)=&\, \int_0^1 \tx_0^{k_r-1}\tx_{\eps_r}\tx_0^{k_{r-1}-1}\tx_{\eps_{r-1}\eps_r}\cdots\tx_0^{k_1-1}\tx_{\eps_1\eps_2\dots \eps_r},\\
\int_0^1 \tx_0^{k_r-1}\tx_{\eps_r}\cdots\tx_0^{k_2-1}\tx_{\eps_2}\tx_0^{k_1-1}\tx_{\eps_1}=&\,
\zeta(k_1,\dots,k_r;\eps_1\eps_2, \dots, \eps_{r-1}\eps_r,\eps_r).
\end{split}
\end{equation*}
Here, we use the fact that $\eps_j=\pm 1$ so that $\eps_j/\eps_{j-1}=\eps_j\eps_{j-1}$.
Similarly, for MMVs we set
\begin{equation*}
 \omz(t)=\frac{dt}{t}, \quad \omn(t)=\frac{2dt}{1-t^2}=\txp(t)-\txn(t), \quad \omp(t)=\frac{2tdt}{1-t^2}=\txp(t)+\txn(t).
\end{equation*}
It is not hard to see that for every admissible composition $\bfk\in\N^r$ and $\bfeps\in\{\pm 1\}^r$,
\begin{equation}\label{equ:MMVint}
\begin{split}
M(k_1,\dots,k_r;\eps_1,\eps_2,\dots,\eps_r)=&\, \int_0^1 \omz^{k_r-1}\om_{\eps_r}\omz^{k_{r-1}-1}\om_{\eps_{r-1}\eps_r}\cdots\omz^{k_1-1}\om_{\eps_1\eps_2\dots \eps_r},\\
\int_0^1 \omz^{k_r-1}\om_{\eps_r}\cdots\omz^{k_2-1}\om_{\eps_2}\omz^{k_1-1}\om_{\eps_1}=&\,
M(k_1,\dots,k_r;\eps_1\eps_2, \dots, \eps_{r-1}\eps_r,\eps_r).
\end{split}
\end{equation}

\subsection{Algebraic set-up for MMVs}\label{subsec:A-MMVs}
For every positive integer $k\ge 1$ and $\eps=\pm1$, define the word of length $k$
$$\z_{k,\eps}:=\omz^{k-1} \om_\eps.$$
One can now define an algebra of words as follows
\begin{defn} \label{defn:fAlevel2}
The set of alphabet $\setX$ consists of the letters $\omz$, $\omp$,
and $\omn$. The \emph{weight} of a word $\bfw$ (\emph{i.e.}, denoted by $|\bfw|$,
a monomial in the letters in $\setX$)
is the number of letters contained in $\bfw$, and its \emph{depth},
denoted by $\dep(\bfw)$, is the number of $\om_\eps$'s contained in $\bfw$.
Define the \emph{MMV} algebra, denoted by $\fA$, to be the (weight) graded
noncommutative polynomial $\Q$-algebra generated by words (including the empty word $\myone$) over the alphabet $\setX$.
Let $\fA^0$ be the subalgebra of $\fA$ generated by words not
beginning with $\om_1$ and not ending with $\omz$. The words in $\fA^0$
are called \emph{admissible words.}
\end{defn}

By Eq.~\eqref{equ:MMVint} every MMV can be expressed as an iterated
integral over the closed interval $[0,1]$ of an admissible word
$\bfw$ in $\fA^0$. This is denoted by
\begin{equation}\label{equ:evaL}
\evaM(\bfw):=\int_0^1 \bfw.
\end{equation}
We also extend $\evaM$ to $\fA$ by $\Q$-linearity.
We remark that the length ${\rm lg}(\bfw)$ of $\bfw$ is equal to the
weight of $\evaM(\bfw)$. Therefore in general for all admissible $\bfk\in\N^r$
and $\bfeps\in\{\pm1\}^r$, from \eqref{equ:MMVint} one has
\begin{align}
 \label{equ:MMV2word}
 M(\bfk;\eps_1,\dots,\eps_r)=&\, \evaM \big(
 \z_{k_r,\eps_r}\z_{k_r,\eps_{r-1}\eps_r}\cdots \z_{k_2,\eps_1\eps_2\dots \eps_r} \big),\\
 \evaM\big(\z_{k_r,\eps_r}\cdots \z_{k_2,\eps_2}\z_{k_1,\eps_1}\big)=&\,
M\big(\bfk;\eps_1\eps_2, \dots, \eps_{r-1} \eps_r, \eps_r \big).
 \label{equ:word2MMV}
\end{align}
Therefore, we can apply Chen's theory of the iterated integrals to compute
the product of two MMVs.

Let $\fA_{\sha}$ be the algebra of $\fA$ where the
multiplication is defined by the shuffle product $\sha$.
Denote the subalgebra $\fA^0$ by $\fA_{\sha}^0$ when one
considers the shuffle product. Then one can easily prove
\begin{pro} \label{prop:MMVshahomo}
The map $\evaML:\fA_{\sha}^0\lra \R$ is an algebra homomorphism.
\end{pro}
\begin{proof}
Similar to the corresponding result for MZVs, this follows easily from Chen's theory of the shuffle product relations of iterated integrals. We leave the details of the proof to the interested reader.
\end{proof}

On the other hand, the MMVs are known to satisfy the
series stuffle relations.
To study such relations in general one has the following
definition.

\begin{defn}
Denote by $\fA^1$ the subalgebra of $\fA$ which is generated by words $\z_{k,\eps}$ with $k\in \N$
and $\eps=\pm1$. Equivalently, $\fA^1$ is the
subalgebra of $\fA$ generated by words not ending with $\omz$. For
any word $\bfw=\z_{k_1,\eps_1}\z_{k_2,\eps_2}\cdots
 \z_{k_r,\eps_r}\in \fA^1$ and $\eps=\pm1$, one
defines the exponent shifting operator $\tau_\eps$ by
\begin{equation*}
\tau_\eps(\bfw)=\z_{k_1,\eps\eps_1}\z_{k_2,\eps\eps_2}\cdots \z_{k_r,\eps\eps_r}.
\end{equation*}
For convenience, on the empty word we adopt the convention that
$\tau_\eps(\myone)=\myone.$ We then define another multiplication $*$
on $\fA^1$ by requiring that $*$ distribute over addition, that
$\myone*\bfw=\bfw*\myone=\bfw$ for any word $\bfw$, and that, for any words
$\bfu,\bfv$, $s,t\in\N$ and $\eps,\eta=\pm 1$
\begin{multline} \label{equ:defnstufflelevelN}
 \bfu\z_{s,\eps}*\bfv \z_{t,\eta}= \Big(\tau_\eps\big(\tau_{\eps}(\bfu)*\bfv\z_{t,\eta}\big)\Big)\z_{s,\eps}
 + \Big(\tau_\eta\big( \bfu*\z_{s,\eps}\tau_{\eta}(\bfv)\big)\Big) \z_{t,\eta}\\
 +2\delta(\eps,\eta)\Big(\tau_{\eps}\big(\tau_{\eps}(\bfu)*\tau_{\eps}(\bfv)\big)\Big)\z_{s+t,\eps},
\end{multline}
where $\delta(\eps,\eta)=0$ or 1 is the Kronecker symbol.
This multiplication $*$ is called the \emph{stuffle product} for MMVs.
\end{defn}

If we denote by $\fA_{*}^1$ the algebra $(\fA^1,*)$ then it is not
hard to prove the next proposition.
\begin{pro}\label{prop:MMVsthomo}
The polynomial algebra $\fA_{*}^1$ is a commutative weight graded $\Q$-algebra.
\end{pro}
\begin{proof}
 A similar proof one can refer to \cite[Theorem 3.2]{H1997}.%See Exercise~\ref{exer:sthomo}.
\end{proof}

Now we can define the subalgebra $\fA_{*}^0$ similar to $\fA_{\sha}^0$
by replacing the shuffle product by the stuffle product. Then by the
induction on the lengths and using the series definition one can
quickly check that for any $\bfw_1,\bfw_2\in \fA_{*}^0$
$$\evaM(\bfw_1)\evaM(\bfw_2)=\evaM(\bfw_1\ast \bfw_2).$$
This implies  the following result.
\begin{pro} \label{prop:MMVstIShomo}
The map $\evaML: \fA_{*}^0 \lra \R$ is an algebra homomorphism.
\end{pro}
\begin{proof}
 A similar proof one can refer to \cite[Theorem 4.2]{H1997}.%See Exercise~\ref{exer:sthomo}.
\end{proof}

\begin{defn}
Let $w$ be an integer such that $w\ge 2$. For nontrivial words
$\bfw_1,\bfw_2\in \fA^0$ with $|\bfw_1|+|\bfw_2|=w$, we say that
the equation
\begin{equation}\label{equ:FDSw}
 \evaM(\bfw_1\sha \bfw_2-\bfw_1*\bfw_2)=0
\end{equation}
provides a \emph{finite double shuffle relation} (finite DBSF) of MMVs
of weight $w$.
\end{defn}

\subsection{Regularization for divergent MMVs}
It is known that even at level one (i.e., the MZV case) these
relations are not enough to provide all the relations among the MZVs.
For example, the weight of the product of two MZVs is at least 4. So
the finite DBSFs cannot imply the well-known identity $\zeta(2,1)=\zeta(3)$.
However, it is believed that one can remedy this by
considering \emph{regularized double shuffle relation} (regularized DBSF)
produced by the following mechanism.

First, combining Prop.~\ref{prop:MMVshahomo} and Prop.~\ref{prop:MMVstIShomo}
we can easily prove the following algebraic result.

\begin{pro} \label{prop:extendM*}
We have an algebra homomorphism:
\begin{equation*}
\evaML_*: (\fA_{*}^1,*)\lra \R[T]
\end{equation*}
which is uniquely determined by the properties that they both
extend the evaluation map $\evaML:\fA^0\lra \R$ by sending
$\omp=\z_{1,1}$ to $T$ and $\omn=\z_{1,-1}$ to $T+2\log 2$.
\end{pro}

\begin{pro} \label{prop:extendMsha}
We have an algebra homomorphism:
\begin{equation*}
\evaML_\sha: (\fA_{\sha}^1,\sha)\lra \R[T]
\end{equation*}
which is uniquely determined by the properties that it
extends the evaluation map $\evaML:\fA^0\lra \R$ by sending
$\omp$ to $T-\log 2$ and $\omn$ to $T+\log 2$.
\end{pro}

\begin{re}
By abuse of notation, we also write $\evaM_\sha(\bfw)$ as $\evaM_\sha(\bfk;\bfeps;T)$
for all $\bfw=\z_{k_r,\eps_r}\dots \z_{k_1,\eps_1}\in\fA^1$ and
similarly for $\evaM_*$. For admissible $\bfw$, we write $\evaM_\sha(\bfk;\bfeps;T)
=\evaM_*(\bfk;\bfq(\bfeps);T)$ simply as $M(\bfw)$,
where $\bfq(\eps_1,\dots,\eps_r)=(\eps_1\eps_2,\dots,\eps_{r-1}\eps_r ,\eps_r )$.
\end{re}

We can now apply the same mechanism as in \cite{IKZ2006} to derive the following regularized DBSFs.

\begin{thm}\label{thm:RDSoverR}
Define an $\R$-linear map $\rho:\R[T]\to \R[T]$ by
 $$\rho(e^{Tu})=\exp\left(\sum_{n=2}^\infty
 \frac{(-1)^n}{n}\zeta(n)u^n\right)e^{(T-\log 2)u},\qquad |u|<1.$$
Then for any $\bfk\in\N^r$ and $\bfeps\in\{\pm1\}^r$ one has
\begin{equation}\label{equ:RDSthm}
 \evaML_\sha(\bfk;\bfeps;T)= \rho\big(\evaML_*(\bfk;\bfq(\bfeps);T) \big).
\end{equation}
\end{thm}

A natural question now arises: for any admissible MMV $M$ is $\log(2)M$ still in the MMV world, i.e., can $\log(2)M$ be expressed as a $\Q$-linear combination of MMVs? We will answer this question in the last section.

\begin{exa} Let's consider a weight three example using a regularized MMV. We have
\begin{align*}
M_\sha(\check{2},\check{1})=&\, M_\sha(\z_{1,-1}\z_{1,-2})= M_\sha(\omn\omz\omn)\\
=&\,  M_\sha(\omn\sha \omz\omn-2 \omz\omn\omn)=(T+\log 2)M(\check{2})-2M_\sha(\check{1},2).
\end{align*}
On the other hand,
\begin{align*}
M_*(\check{2},1)=&\, M_*\big(\bfq(\z_{1,-1}\z_{1,-2})\big)=M_*(\z_{1,1}\z_{1,-2})\\
 =&\, M_*(\z_{1,1}* \z_{1,-2} -  \z_{1,-2}\z_{1,1})=TM(\check{2})-M_*(1,\check{2}).
\end{align*}
By Theorem~\ref{thm:RDSoverR}, we get
\begin{align*}
(T+\log 2)M(\check{2})-2M(\check{1},2)=&\, \rho\Big(TM(\check{2})-M(1,\check{2})\Big)=(T-\log 2)M(\check{2})-M(1,\check{2}) \\
&\, \Lra 2M(\check{1},2)=2\log(2) M(\check{2})+M(1,\check{2}) \Big(=\frac{7}{2}\zeta(3)\Big).
\end{align*}
\end{exa}

The example above is a particular case of the following result. The more general result
will be proved in the last section, but unfortunately requires much more complicated mechanism.

\begin{pro} Let $\eps=\pm1$ and $\bfk$ be an admissible composition of positive integers of weight $w$.
Then $\log(2) M(\bfk;1,\dots,1,\eps)$ can be
expressed as a $\Q$-linear combination of MMVs of weight $w+1$.
\end{pro}
\begin{proof} For convenience, given any two polynomials $M_1(T)$ and $M_2(T)$
we denote by $M_1\equiv M_2$ if their difference $M_1-M_2$ can be
expressed as a $\Q$-linear combination of MMVs of weight $w+1$.
Assume $\eps=-1$ first. Then we have
\begin{align*}
M_\sha( z_{1,-1} z_{k_r,-1}z_{k_{r-1},1}\dots z_{k_1,\eps_1})
\equiv &\,  M_\sha(\omn\sha z_{k_r,-1}z_{k_{r-1},1}\dots z_{k_1,\eps_1})
=(T+\log 2)M(\bfk;\bfeps).
\end{align*}
On the other hand,
\begin{align*}
M_*\big(\bfq(z_{1,-1} z_{k_r,-1}z_{k_{r-1},1}\dots z_{k_1,\eps_1})\big)= &\,
M_*(z_{1,1} z_{k_r,-1}z_{k_{r-1},1}\dots z_{k_1,\eps_1}) \\
\equiv &\,
   M_*(\z_{1,1}* z_{k_r,-1}z_{k_{r-1},1}\dots z_{k_1,\eps_1})=TM(\bfk;\bfeps).
\end{align*}
By Theorem~\ref{thm:RDSoverR}, we get
\begin{align*}
(T+\log 2)M(\bfk;-1,1,\dots,1) \equiv
(T-\log 2)M(\bfk;-1,1,\dots,1).
\end{align*}
Thus $\log(2) M(\bfk;-1,1,\dots,1)\equiv 0$.

Now if $\eps=1$ then the same idea as above shows that
\begin{align*}
M_\sha( z_{1,1} z_{k_r,1}\dots z_{k_1,\eps_1}) \equiv &\, (T-\log 2)M(\bfk;1,\dots,1), \\
M_*\big(\bfq(z_{1,1} z_{k_r,1}\dots z_{k_1,\eps_1})\big) \equiv &\ TM(\bfk;1,\dots,1).
\end{align*}
Thus $\log(2) M(\bfk;-1,1,\dots,1)\equiv 0$ by Theorem~\ref{thm:RDSoverR} again.
This completes the proof of the proposition.
\end{proof}

\begin{thm} \label{thm:dualMMVo}
Let $\bfk\in\N^r$ be an admissible composition and $\bfeps\in\{\pm 1\}^r$ with $\eps_1=-1$.
Then
\begin{equation}\label{equ:dualMMVo}
 M \big(\omn^{l_1} \omz^{k_1}\om_{\eps_2}^{l_2} \omz^{k_2} \cdots\om_{\eps_r}^{l_r} \omz^{k_r}  \big)
=M \big(\omn^{k_r} u_{\eps_r}^{l_r}  \cdots \omn^{k_2} u_{\eps_2}^{l_1}  \omn^{k_1} \omz^{l_1}  \big),
\end{equation}
where $u_{-1}=\omz$ and $u_1=\omz+\omp-\omn$.
\end{thm}
\begin{proof}
This follows immediately from the substitution $t\to \frac{1-t}{1+t}$. We leave the detail to the interested reader.
\end{proof}

\begin{exa} In weight 4, we have the duality relation
\begin{equation*}
    M(\check{1},1,2)=\int_0^1 \omz\omp\omn^2=\int_0^1 \omz^2(\omz+\omp-\omn)\omn=M(\check{4})+M(\check{1},\check{3})-M(\check{1},3).
\end{equation*}
\end{exa}

\begin{re} We know there should be $\Q$-linear relations among Euler sums that are not consequences
of the finite DBSF, see \cite[Remark 3.5]{Zhao2010a}. The same should hold for MMVs.
\end{re}

We end this section by the following parity theorem of MMVs of arbitrary depth which follows from the general parith result of Panzer on colored MZVs (see \cite{Panzer2017}).
\begin{thm}\label{thm:MMVparity}
Let $\bfk$ be an admissible composition and assume its depth $r$ and weight $w$ are of different parity. Then $M(\bfk;\bfeps)$ can be expressed as a $\Q$-linear combination of multiple $M$-values of lower depths and products of multiple $M$-values with sum of depths smaller than $r$.
\end{thm}

\section{Kaneko-Tsumura $\psi$-Values and Convoluted $T$-Values}\label{sec:KTV}
In this section, we establish some explicit formulas involving Kaneko-Tsumura $\psi$-values and the convoluted $T$-values. First, we evaluate a few families of integrals involving the natural logarithm in terms of MTVs and MSVs. Recall that $\{l\}_m$ means the sequence $\underbrace{l,\ldots,l}_{m \text{\;times}}$.

\begin{thm}\label{thm-I} For positive integers $m$ and $n$, the following identities hold.
\begin{align}
&\begin{aligned}
\int_{0}^1 t^{2n-2} \log^{2m}\tt dt&= \frac{2(2m)!}{2n-1} \sum_{j=0}^m {\bar \zeta}(2m-2j)T_n(\{1\}_{2j}),\label{ee}
\end{aligned}\\
&\begin{aligned}
\int_{0}^1 t^{2n-2} \log^{2m-1}\tt dt&= -\frac{2(2m-1)!}{2n-1} \sum_{j=0}^{m-1} {\bar \zeta}(2m-1-2j)T_n(\{1\}_{2j})\\&\quad-\frac{(2m-1)!}{2n-1} S_n(\{1\}_{2m-1}),\label{eo}
\end{aligned}\\
&\begin{aligned}
\int_{0}^1 t^{2n-1} \log^{2m}\tt dt&=\frac{(2m)!}{n} \sum_{j=0}^{m-1} {\bar \zeta}(2m-1-2j)T_n(\{1\}_{2j+1})\\&\quad+\frac{(2m)!}{2n} S_n(\{1\}_{2m}),\label{oe}
\end{aligned}\\
&\begin{aligned}
\int_{0}^1 t^{2n-1} \log^{2m-1}\tt dt&= -\frac{(2m-1)!}{n} \sum_{j=0}^{m-1} {\bar \zeta}(2m-2-2j)T_n(\{1\}_{2j+1}),\label{oo}
\end{aligned}
\end{align}
where ${\bar \zeta}(m):=-\zeta(\overline{ m})$, and ${\bar \zeta}(0)$ should be interpreted as $1/2$ wherever it occurs..
\end{thm}
\begin{proof} Consider integral
\begin{align*}
&\int_{0}^x t^{2n-2}\log^m\tt dt\\&=\frac{x^{2n-1}}{2n-1} \log^m\xx +\frac{2m}{2n-1} \int_{0}^x \frac{t^{2n-1}}{1-t^2} \log^{m-1}\tt dt\\
&=\frac{x^{2n-1}}{2n-1} \log^m\xx-\frac{2m}{2n-1} \int_{0}^x \frac{t-t^{2n-1}}{1-t^2} \log^{m-1}\tt dt\\&\quad +\frac{2m}{2n-1} \int_{0}^x \frac{t}{1-t^2} \log^{m-1}\tt dt\\
&=\frac{x^{2n-1}}{2n-1} \log^m\xx-\frac{2m}{2n-1} \sum_{k=1}^{n-1}\int_{0}^xt^{2k-1}\log^{m-1}\tt dt\\&\quad +\frac{2m}{2n-1} \int_{0}^x \frac{t}{1-t^2} \log^{m-1}\tt dt.
\end{align*}
Applying $t=\frac{1-u}{1+u}$ in last integral on the right-hand side of above, we evaluate
\begin{align*}
&\int_{0}^x \frac{t}{1-t^2} \log^{m-1}\tt dt\\&=\int_{(1-x)/(1+x)}^1 \frac{1-u}{2u(1+u)} \log^{m-1}(u)du\\
&=\frac1{2}\int_{(1-x)/(1+x)}^1 \frac{\log^{m-1}(u)}{u}du-\int_{(1-x)/(1+x)}^1 \frac{\log^{m-1}(u)}{1+u}du\\
&=(-1)^{m-1} (m-1)! \zeta(\overline {m})- \frac{1}{2m} \log^{m}\xx \\&\quad+\sum_{l=0}^{m-1} l!\binom{m-1}{l} (-1)^{l+1}\log^{m-1-l}\xx {\rm Li}_{l+1} \left( \frac{x-1}{x+1}\right),
\end{align*}
where we used the well-know identities
\begin{align*}
\int\limits_0^x {{t^{n - 1}}{{\left( {\log (t)} \right)}^m}} dt = \sum\limits_{l = 0}^m {l!\binom{m}{l}\frac{{{{\left( { - 1} \right)}^l}}}{{{n^{l + 1}}}}{{\left( {\log (x)} \right)}^{m - l}}{x^n}},\quad x\in (0,1).
\end{align*}
Hence, one obtain
\begin{align*}
&\int_{0}^x t^{2n-2}\log^m\tt dt\\&=\frac{x^{2n-1}}{2n-1} \log^m\xx +\frac{2m}{2n-1} \int_{0}^x \frac{t^{2n-1}}{1-t^2} \log^{m-1}\tt dt\\
&=\frac{x^{2n-1}-1}{2n-1} \log^m\xx-\frac{2m}{2n-1} \sum_{k=1}^{n-1}\int_{0}^xt^{2k-1}\log^{m-1}\tt dt\\&\quad -\frac{2m!(-1)^{m}}{2n-1}\zeta(\overline {m})+\frac{2}{2n-1}\sum_{l=1}^{m} l!\binom{m}{l} (-1)^{l}\log^{m-l}\xx {\rm Li}_{l} \left( \frac{x-1}{x+1}\right).
\end{align*}
Similarly, we deduce
\begin{align*}
&\int_{0}^x t^{2n-1}\log^m\tt dt\\
&=\frac{x^{2n}-1}{2n} \log^m \xx -\frac{m}{n} \sum_{k=1}^n \int_{0}^x t^{2k-2} \log^{m-1}\tt dt.
\end{align*}
Then, letting $x$ tends to $1$, we can get the following two recurrence relations
\begin{align*}
&\int_{0}^1 t^{2n-2}\log^m\tt dt=-\frac{2m!(-1)^{m}}{2n-1}\zeta(\overline {m})-\frac{2m}{2n-1} \sum_{k=1}^{n-1}\int_{0}^1 t^{2k-1}\log^{m-1}\tt dt,\\
&\int_{0}^1 t^{2n-1}\log^m\tt dt=-\frac{m}{n} \sum_{k=1}^n \int_{0}^1 t^{2k-2} \log^{m-1}\tt dt.
\end{align*}
Thus, from above recurrence formulas, we may deduce these desired evaluations.\end{proof}

It is possible that the closed form of integral $\int_{0}^1 t^{n-1} \log^p(t) \log^m \tt dt$ can be proved by
using the techniques of the Theorem \ref{thm-I}. Two examples as follow
\begin{align*}
\int_{0}^1 t^{2n-2} \log(t) \log^2 \tt dt&=-\frac{2}{2n-1}T_n(1,2)-\frac{2}{2n-1}T_n(2,1)-\frac{2}{(2n-1)^2}T_n(1,1)\\&\quad+\frac{3\zeta(2)}{2n-1}S_n(1)-\frac{2\zeta(2)}{(2n-1)^2}-\frac{7\zeta(3)}{2(2n-1)}+\frac{6\log(2)\zeta(2)}{2n-1},\\
\int_{0}^1 t^{2n-1} \log(t) \log^2 \tt dt&=-\frac{1}{n} S_n(1,2)-\frac{1}{n} S_n(2,1)-\frac{1}{2n^2} S_n(1,1)-\frac{2\log(2)}{n} T_n(2)\\&\quad+\frac{\zeta(2)}{2n} T_n(1)-\frac{\log(2)}{n^2} T_n(1)+\frac{7\zeta(3)}{4n}.
\end{align*}

\begin{re} A similar result of Theorem \ref{thm-I} can be found in \cite{Xu2017}:
\begin{align}\label{b5}
\int_0^1 x^{n-1}\log^m(1-x)dx=(-1)^mm!\frac{\zeta^\star_n(\{1\}_{m})}{n},
\end{align}
where $m\geq 0$ and $n\geq 1$ are positive integers.
\end{re}

Now, we prove four explicit formulas for Kaneko-Tsumura $\psi$-values via the convoluted $T$-values $T({\bfk_{r}}\circledast {\bfl_{s}})$. Recall from \cite{KTA2018,KTA2019} that the Kaneko-Tsumura $\psi$-function is defined by
\begin{align}\label{a14}
\psi(k_1,k_2\ldots,k_r;s):=\frac{1}{\Gamma(s)} \int\limits_{0}^\infty \frac{t^{s-1}}{\sinh(t)}{\rm A}({k_1,k_2,\ldots,k_r};\tanh(t/2))dt\quad (\Re(s)>0),
\end{align}
where
\begin{align}\label{a15}
&{\rm A}(k_1,k_2,\ldots,k_r;z): = 2^r\sum\limits_{1 \le {n_1} <  \cdots  < {n_r}\atop n_i\equiv i\ {\rm mod}\ 2} {\frac{{{z^{{n_r}}}}}{{n_1^{{k_1}}n_2^{{k_2}} \cdots n_r^{{k_r}}}}},\quad z \in \left[ { - 1,1} \right).
\end{align}
Obviously, the Kaneko-Tsumura $\psi$-function can be regarded as a new kind of Arakawa-Kaneko zeta function of level two. The Arakawa-Kaneko zeta function (see \cite{AM1999}) is defined by
\begin{align}
\xi(k_1,k_2\ldots,k_r;s):=\frac{1}{\Gamma(s)} \int\limits_{0}^\infty \frac{t^{s-1}}{e^t-1}{\rm Li}_{k_1,k_2,\ldots,k_r}(1-e^{-t})dt\quad (\Re(s)>0),
\end{align}
where ${\mathrm{Li}}_{{{k_1},{k_2}, \cdots ,{k_r}}}\left( z \right)$ is the single variable multiple polylogarithm function defined by
\begin{align}\label{a17}
&{\mathrm{Li}}_{{{k_1},{k_2}, \cdots ,{k_r}}}\left( z \right): = \sum\limits_{1 \le {n_1} <  \cdots  < {n_r}} {\frac{{{z^{{n_r}}}}}{{n_1^{{k_1}}n_2^{{k_2}} \cdots n_r^{{k_r}}}}},\quad z \in \left[ { - 1,1} \right).
\end{align}

\begin{thm}\label{thmb3} For positive integers $m$ and $p$,
\begin{align}
&\begin{aligned}
\psi(\bfk_{2m-1};2p)=2\sum_{j=0}^{p-1}{\bar \zeta}(2p-1-2j)T(\bfk_{2m-1}\circledast \{1\}_{2j+1})+T(\bfk_{2m-1}\circledast \{1\}_{2p}),\label{b6}
\end{aligned}\\
&\begin{aligned}
\psi(\bfk_{2m-1};2p+1)=2\sum_{j=0}^{p}{\bar \zeta}(2p-2j)T(\bfk_{2m-1}\circledast \{1\}_{2j+1}),\label{b7}
\end{aligned}\\
&\begin{aligned}
\psi(\bfk_{2m};2p)=2\sum_{j=0}^{p-1}{\bar \zeta}(2p-2-2j)T(\bfk_{2m}\circledast \{1\}_{2j+2}),\label{b8}
\end{aligned}\\
&\begin{aligned}
\psi(\bfk_{2m};2p+1)=2\sum_{j=0}^{p-1}{\bar \zeta}(2p-1-2j)T(\bfk_{2m}\circledast \{1\}_{2j+2})+T(\bfk_{2m}\circledast \{1\}_{2p+1}),\label{b9}
\end{aligned}
\end{align}
where ${\bar \zeta}(0)$ should be interpreted as $1/2$ wherever it occurs.
\end{thm}
\begin{proof} From the definition of Kaneko-Tsumura $\psi$-function \eqref{a14},
we set variables $\tanh(t)=x$ and $s=p+1$ to get
\begin{align}\label{b10}
\psi(k_1,k_2\ldots,k_r;p+1)=\frac{(-1)^p}{p!}\int\limits_{0}^1 \frac{\log^p\left(\frac{1-x}{1+x}\right){\rm A}(k_1,k_2,\ldots,k_r;x)}{x}dx.
\end{align}
Then, according to the definitions of ${\rm A}(k_1,k_2,\ldots,k_r;x)$ and MTHSs, by an elementary calculation, we can find that
\begin{align}\label{b11}
&{\rm A}(\bfk_{2m-1};x)=2\su \frac{T_n(\bfk_{2m-2})}{(2n-1)^{k_{2m-1}}}x^{2n-1}
\end{align}
and
\begin{align}\label{b12}
&{\rm A}(\bfk_{2m};x)=2\su \frac{T_n(\bfk_{2m-1})}{(2n)^{k_{2m}}}x^{2n}.
\end{align}
Hence, applying Theorem \ref{thm-I} and identities \eqref{b11} and \eqref{b12} to \eqref{b10}, we may easily deduce these desired results.\end{proof}

\begin{re} A similar evaluation of Theorem \ref{thmb3} for Arakawa-Kaneko zeta values can be found in \cite{Ku2010},
\begin{align*}
\xi(\bfk;p+1)=\su \frac{\zeta_{n-1}(k_1,\ldots,k_{r-1})\zeta^\star_n(\{1\}_{p})}{n^{k_r+1}}.
\end{align*}
\end{re}

\begin{cor}\label{corb4} For positive integers $p$,
\begin{align}
&\psi(k;2p)=2\sum_{j=0}^{p-1} {\bar \zeta}(2p-1-2j)T(\{1\}_{2j},k+1)+S(\{1\}_{2p-1},k+1),\label{b13}\\
&\psi(k;2p+1)=2\sum_{j=0}^p {\bar \zeta}(2p-2j)T(\{1\}_{2j},k+1).\label{b14}
\end{align}
\end{cor}
\begin{proof} Corollary \ref{corb4} follows immediately from \eqref{b6} and \eqref{b7} by setting $m=1$ and $k_1=k$.\end{proof}

In \cite{KTA2018}, Kaneko and Tsumura proved the following formula ($k,r\geq 1$ and $p\geq 0$ are integers)
\begin{align}\label{b15}
\psi(\{1\}_{r-1},k;p+1)=\sum_{i_1+\cdots+i_k=p,\atop i_1,\ldots,i_k\geq 0} \binom{i_k+r}{r} T(i_1+1,\ldots,i_{k-1}+1,i_k+r+1).
\end{align}
Therefore, letting $r=1$ in \eqref{b15} and using \eqref{b13}-\eqref{b14}, we can find many relations between MTVs and MSVs. For example, setting $k=1$ in \eqref{b13} yields
\begin{align}\label{b16}
S(\{1\}_{2p-1},2)=2pT(2p+1)-2\sum_{j=0}^{p-1} {\bar \zeta}(2p-1-2j)T(2j+2),
\end{align}
where we used the duality relation of MTVs
$$T(\{1\}_{r-1},k+1)=T(\{1\}_{k-1},r+1).$$
Setting $p=2$ in \eqref{b16} yields
\[S(1,1,1,2)=-\frac{15}{4}\log(2)\zeta(4)-\frac{9}{4}\zeta(2)\zeta(3)+\frac{31}{4}\zeta(5).\]
In next section, we will prove that all Kaneko-Tsumura $\psi$-values can be expressed in terms of MTVs.

Similarly, according to the definition of multiple polylogarithm function \eqref{a17} and using \eqref{oe} and \eqref{oo}, we prove the following theorem.
\begin{thm} For a composition $\bfk=(k_1,\ldots,k_r)$ and positive integer $m$,
\begin{align}\label{b17}
\int_{0}^1 \frac{{\rm Li}_{\bfk}(x^2)\log^{2m}\xx}{x}dx&=(2m)!\sum_{j=0}^{m-1} {\bar \zeta}(2m-1-2j) \su \frac{\zeta_{n-1}(\bfk_{r-1})T_n(\{1\}_{2j+1})}{n^{k_r+1}}\nonumber\\&\quad +\frac{(2m)!}{2} \su \frac{\zeta_{n-1}(\bfk_{r-1})S_n(\{1\}_{2m})}{n^{k_r+1}}
\end{align}
and
\begin{align}\label{b18}
\int_{0}^1 \frac{{\rm Li}_{\bfk}(x^2)\log^{2m-1}\xx}{x}dx=-(2m-1)!\sum_{j=0}^{m-1} {\bar \zeta}(2m-2-2j) \su \frac{\zeta_{n-1}(\bfk_{r-1})T_n(\{1\}_{2j+1})}{n^{k_r+1}}.
\end{align}
\end{thm}

Note that if $\bfk=(\{1\}_r)$ then the integrals on the left-hand side on the \eqref{b17} and \eqref{b18} can be evaluated by alternating MZVs. Now, we get explicit formulas. Let
\[I(r,p):=\int_0^1 \frac{\log^r(1-x^2)\log^p\xx}{x}dx,\quad r,p\geq 1.\]
Change $x=(1-u)/(1+u)$ to give
\begin{align}\label{b19}
I(r,p)&=2\int_0^1 \frac{\log\left(\frac{4u}{(1+u)^2}\right)\log^p(u)}{1-u^2}du\nonumber\\
&=2\sum_{i+j+k=r,\atop i,j,k\geq 0} \frac{r!}{i!j!k!} \log^i(4)\log^j(u)(-\log(1+u)^2)^k \frac{\log^p(u)}{1-u^2}du\nonumber\\
&=\sum_{i+j+k=r,\atop i,j,k\geq 0} \frac{r!}{i!j!k!}  2^{k+i+1} (-1)^k \log^i(2) \int_0^1 \frac{\log^{p+j}(u)\log^k(1+u)}{1-u^2}du.
\end{align}
Applying the two well-known identities
\begin{align*}
\frac{\log^k(1+u)}{1-u}=(-1)^kk!\su \zeta_{n-1}(\{1\}_{k-1},{\bar 1})u^{n-1},\quad \zeta_{n-1}(\{1\}_{-1},{\bar 1}):=1,
\end{align*}
and
\begin{align*}
\frac{\log^k(1+u)}{1+u}=(-1)^kk!\su \zeta_{n-1}(\{1\}_{k})(-u)^{n-1}
\end{align*}
to \eqref{b19} yields
\begin{align}\label{b20}
I(r,p)=\sum_{i+j+k=r,\atop i,j,k\geq 0} \frac{r!(p+j)!}{i!j!}(-1)^{p+j} 2^{k+i} \log^i(2)\left(\zeta(\{1\}_{k-1},{\bar 1},p+j+1)-\zeta(\{1\}_{k},\overline{p+j+1}) \right),
\end{align}
where $\zeta(\{1\}_{-1},{\bar 1},p+j+1):=\zeta(p+j+1)$.

Hence, setting $\bfk=(\{1\}_r)$ in \eqref{b17} and \eqref{b18}, and using \eqref{b20}, we can obtain the following corollaries immediately.
\begin{cor} For positive integer $m$ and $r$,
\begin{align}\label{b21}
&2\sum_{j=0}^{m-1} {\bar \zeta}(2m-1-2j) \su \frac{\zeta_{n-1}(\{1\}_{r-1})T_n(\{1\}_{2j+1})}{n^{2}} +\su \frac{\zeta_{n-1}(\{1\}_{r-1})S_n(\{1\}_{2m})}{n^{2}}\nonumber\\
&=(-1)^r\sum_{i+j+k=r,\atop i,j,k\geq 0}(-1)^j 2^{k+i+1}\binom{2m+j}{j} \frac{\log^i(2)}{i!}\left(\zeta(\{1\}_{k-1},{\bar 1},2m+j+1)-\zeta(\{1\}_{k},\overline{2m+j+1}) \right).
\end{align}
\end{cor}

\begin{cor} For positive integer $m$ and $r$,
\begin{align}\label{b22}
&2\sum_{j=0}^{m-1} {\bar \zeta}(2m-2-2j) \su \frac{\zeta_{n-1}(\{1\}_{r-1})T_n(\{1\}_{2j+1})}{n^{2}} \nonumber\\
&=(-1)^r\sum_{i+j+k=r,\atop i,j,k\geq 0}(-1)^j 2^{k+i}\binom{2m+j-1}{j} \frac{\log^i(2)}{i!}\left(\zeta(\{1\}_{k-1},{\bar 1},2m+j)-\zeta(\{1\}_{k},\overline{2m+j}) \right).
\end{align}
\end{cor}

Therefore, from two corollaries above, we know that for positive integers $r$ and $m$, the series
\[\su \frac{\zeta_{n-1}(\{1\}_{r-1})T_n(\{1\}_{2m-1})}{n^{2}}\quad\text{and}\quad\su \frac{\zeta_{n-1}(\{1\}_{r-1})S_n(\{1\}_{2m})}{n^{2}}\]
can be evaluated in terms of alternating MZVs.
In particular, setting $r=1$ in \eqref{b21} and \eqref{b22} one obtains
\begin{align*}
S(\{1\}_{2m},2)&=2\zeta(1,\overline{2m+1})-2\log(2)T(2m+1)-2\zeta({\bar 1},2m+1) \\
&\quad+(2m+1)T(2m+2)-2\sum_{j=0}^{m-1}{\bar \zeta}(2m-1-2j)T(2j+3)
\end{align*}
and
\begin{align*}
\sum_{j=0}^{m-1} {\bar \zeta}(2m-2-2j)T(2j+3)=\zeta(1,\overline{2m})-\log(2)T(2m)-\zeta({\bar 1},2m)+mT(2m+1).
\end{align*}

Finally, we end this section by a result relating some convoluted values involving MTVs, MSVs and MZVs.
\begin{thm} For positive integers $k,m$ and $r$,
\begin{align}
&2\sum_{j=0}^{m-1} {\bar \zeta}(2m-1-2j) \su \frac{\zeta_{n-1}(\{1\}_{r-1})T_n(\{1\}_{2j+1})}{n^{k+1}} +\su \frac{\zeta_{n-1}(\{1\}_{r-1})S_n(\{1\}_{2m})}{n^{k+1}}\nonumber\\
&+(-1)^k \su \frac{T_n(\{1\}_{2m-1})\zeta^\star_n(\{1\}_r)}{n^{k+1}}=\sum_{j=1}^{k-1} (-1)^{j-1}2^j T(\{1\}_{2m-1},j+1)\zeta(\{1\}_{r-1},k+1-j).
\end{align}
\end{thm}
\begin{proof} On the one hand, setting $\bfk=(\{1\}_{r-1},k)$ in \eqref{b17} yields
\begin{align*}
&\int_0^1 \frac{{\rm Li}_{\{1\}_{r-1},k}(x^2)\log^{2m}\xx}{x}dx\\
&=(2m)!\sum_{j=0}^{m-1} {\bar \zeta}(2m-1-2j) \su \frac{\zeta_{n-1}(\{1\}_{r-1})T_n(\{1\}_{2j+1})}{n^{k+1}} +\frac{(2m)!}{2}\su \frac{\zeta_{n-1}(\{1\}_{r-1})S_n(\{1\}_{2m})}{n^{k+1}}.
\end{align*}
On the other hand, according to definitions, we have
\begin{align*}
(-1)^rr!{\rm Li}_{\{1\}_{r-1},k}(x^2)=2^{k-1} \int\limits_{0<t_1<\cdots<t_{k-1}<x} \frac{\log^r(1-t_1^2)}{t_1\cdots t_{k-1}}dt_1\cdots t_{k-1}
\end{align*}
and
\begin{align*}
(-1)^rr!{\rm A}(\{1\}_{r-1},k;x)= \int\limits_{0<t_1<\cdots<t_{k-1}<x} \frac{\log^r\left(\frac{1-t_1}{1+t_1}\right)}{t_1\cdots t_{k-1}}dt_1\cdots t_{k-1}.
\end{align*}
Hence, using integration by parts, by a direct calculation, we deduce
\begin{align*}
&\int_0^1 \frac{{\rm Li}_{\{1\}_{r-1},k}(x^2)\log^{2m}\xx}{x}dx\\
&= (2m)!\sum_{j=1}^{k-1} (-1)^{j-1}2^{j-1} T(\{1\}_{2m-1},j+1)\zeta(\{1\}_{r-1},k+1-j)\\
&\quad+\frac{(-1)^{k+r-1}}{r!} 2^{k-1} (2m)! \int_0^1 \frac{\log^r(1-x^2){\rm A}(\{1\}_{2m-1},k;x)}{x}dx\\
&=(2m)!\sum_{j=1}^{k-1} (-1)^{j-1}2^{j-1} T(\{1\}_{2m-1},j+1)\zeta(\{1\}_{r-1},k+1-j)\\
&\quad+\frac{(-1)^{k+r-1}}{r!} 2^{k} (2m)! \su \frac{T_n(\{1\}_{2m-1})}{(2n)^k} \int_0^1 x^{2n-1}\log^r(1-x^2)dx\\
&=(2m)!\sum_{j=1}^{k-1} (-1)^{j-1}2^{j-1} T(\{1\}_{2m-1},j+1)\zeta(\{1\}_{r-1},k+1-j)\\
&\quad -(-1)^k \frac{(2m)!}{2} \su \frac{T_n(\{1\}_{2m-1})\zeta^\star_n(\{1\}_r)}{n^{k+1}},
\end{align*}
where we used the \eqref{b5}. Thus, we complete this proof.\end{proof}

Clearly, if $k=1$ then
\begin{align}
\su \frac{T_n(\{1\}_{2m-1})\zeta^\star_n(\{1\}_r)}{n^{2}}=\frac{2(-1)^r}{r!(2m)!}I(r,2m).
\end{align}
if $r=1$ then
\begin{align}
&2^{k+1}\sum_{j=0}^{m-1} {\bar \zeta}(2m-1-2j)T(\{1\}_{2j+1},k+1)+2^kS(\{1\}_{2m},k+1)+(-1)^k \su \frac{T_n(\{1\}_{2m-1})H_n}{n^{k+1}}\nonumber\\
&=\sum_{j=1}^{k-1} (-1)^{j-1}2^{j} T(\{1\}_{2m-1},j+1)\zeta(k+1-j),
\end{align}
where $H_n:=\zeta_n(1)$ is classical harmonic number.

\section{Multiple Associated Integrals and 2-Labeled Posets}\label{sec:poset}
\begin{defn}
A \emph{$2$-poset} is a pair $(X,\delta_X)$, where $X=(X,\leq)$ is
a finite partially ordered set and $\delta_X$ is a map from $X$ to $\{0,1\}$.
We often omit  $\delta_X$ and simply say ``a 2-poset $X$.''
The $\delta_X$ is called the \emph{label map} of $X$.

A 2-poset $(X,\delta_X)$ is called \emph{admissible} if
$\delta_X(x)=0$ for all maximal elements $x\in X$ and
$\delta_X(x)=1$ for all minimal elements $x\in X$.
\end{defn}

\begin{defn}
For an admissible 2-poset $X$, we define the associated integral
\begin{equation}\label{4.1}
I(X)=\int_{\Delta_X}\prod_{x\in X}\om_{\delta_X(x)}(t_x),
\end{equation}
where
\[\Delta_X=\bigl\{(t_x)_x\in [0,1]^X \bigm| t_x<t_y \text{ if } x<y\bigr\}\]
and
\[\omz(t)=\frac{dt}{t}, \quad \omn(t)=\frac{2dt}{1-t^2}. \]
\end{defn}

Note that in \cite{KY2018,KO2018,Y2014}, $\omn(t)=\frac{dt}{1-t}$. For the empty 2-poset, denoted $\emptyset$, we put $I(\emptyset):=1$.

\begin{pro}\label{prop:shuffl2poset}
For non-comparable elements $a$ and $b$ of a $2$-poset $X$, $X^b_a$ denotes the $2$-poset that is obtained from $X$ by adjoining the relation $a<b$. If $X$ is an admissible $2$-poset, then the $2$-poset $X^b_a$ and $X^a_b$ are admissible and
\begin{equation}\label{4.2}
I(X)=I(X^b_a)+I(X^a_b).
\end{equation}
\end{pro}

Note that the admissibility of a 2-poset corresponds to
the convergence of the associated integral. We use Hasse diagrams to indicate 2-posets, with vertices $\circ$ and $\bullet$ corresponding to $\delta(x)=0$ and $\delta(x)=1$, respectively.  For example, the diagram
\[\begin{xy}
{(0,-4) \ar @{{*}-o} (4,0)},
{(4,0) \ar @{-{*}} (8,-4)},
{(8,-4) \ar @{-o} (12,0)},
{(12,0) \ar @{-o} (16,4)}
\end{xy} \]
represents the 2-poset $X=\{x_1,x_2,x_3,x_4,x_5\}$ with order
$x_1<x_2>x_3<x_4<x_5$ and label
$(\delta_X(x_1),\ldots,\delta_X(x_5))=(1,0,1,0,0)$.
This 2-poset is admissible.
To describe the corresponding diagram, we introduce an abbreviation:
For a sequence $\bfk=(k_1,\ldots,k_r)$ of positive integers,
we write
\[\begin{xy}
{(0,-3) \ar @{{*}.o} (0,3)},
{(1,-3) \ar @/_1mm/ @{-} _\bfk (1,3)}
\end{xy}\]
for the vertical diagram
\[\begin{xy}
{(0,-24) \ar @{{*}-o} (0,-20)},
{(0,-20) \ar @{.o} (0,-14)},
{(0,-14) \ar @{-} (0,-10)},
{(0,-10) \ar @{.} (0,-4)},
{(0,-4) \ar @{-{*}} (0,0)},
{(0,0) \ar @{-o} (0,4)},
{(0,4) \ar @{.o} (0,10)},
{(0,10) \ar @{-{*}} (0,14)},
{(0,14) \ar @{-o} (0,18)},
{(0,18) \ar @{.o} (0,24)},
{(1,-24) \ar @/_1mm/ @{-} _{k_1} (1,-14)},
{(4,-3) \ar @{.} (4,-11)},
{(1,0) \ar @/_1mm/ @{-} _{k_{r-1}} (1,10)},
{(1,14) \ar @/_1mm/ @{-} _{k_r} (1,24)}
\end{xy}.\]

Then, by the definition of the multiple $T$-values $T(\bfk)$, we have
\begin{align*}
&T(\bfk)=\int\limits_{0}^1 \underbrace{\frac{dt}{t}\cdots\frac{dt}{t}}_{k_r-1}\frac{2dt}{1-t^2}
\underbrace{\frac{dt}{t}\cdots\frac{dt}{t}}_{k_{r-1}-1}\frac{2dt}{1-t^2}\cdots
\underbrace{\frac{dt}{t}\cdots\frac{dt}{t}}_{k_1-1}\frac{2dt}{1-t^2}.
\end{align*}
Hence, using this notation of associated integral, one can verify that

\begin{equation}\label{4.3}
T(\bfk)=I\left(\ \begin{xy}
{(0,-3) \ar @{{*}.o} (0,3)},
{(1,-3) \ar @/_1mm/ @{-} _\bfk (1,3)}
\end{xy}\right).
\end{equation}
Similarly, using the fact that if $x=\tanh(t/2)$ then $dx/x=dt/\sinh(t)$ and $2dx/(1-x^2)=dt$ we deduce from \cite[Lemma 5.1]{KTA2018} that
\begin{equation*}
\psi(\bfk;p+1)=\frac{1}{p!}\,I\left(\xybox{
{(0,-9) \ar @{{*}-o} (0,-4)},
{(0,-4) \ar @{.o} (0,4)},
{(0,4) \ar @{-o} (10,9)},
{(10,9) \ar @{-{*}} (6,-5)},
{(10,9) \ar @{-{*}} (16,-5)},
{(8,-5) \ar @{.} (14,-5)},
{(-1,-9) \ar @/^1mm/ @{-} ^\bfk (-1,4)},
{(6,-6) \ar @/_1mm/ @{-} _{p} (16,-6)},
}\ \right)=I\left(\xybox{
{(0,-9) \ar @{{*}-o} (0,-4)},
{(0,-4) \ar @{.o} (0,4)},
{(0,4) \ar @{-o} (5,9)},
{(10,-9) \ar @{{*}-{*}} (10,-4)},
{(10,-4) \ar @{.{*}} (10,4)},
{(10,4) \ar @{-} (5,9)},
{(-1,-9) \ar @/^1mm/ @{-} ^\bfk (-1,4)},
{(11,-9) \ar @/_1mm/ @{-} _{p} (11,4)},
}\ \right)
\end{equation*}
since there are exactly $p!$ ways to impose
a total order on the $p$ black vertices.
By Prop.~\ref{prop:shuffl2poset} this implies the result that $\psi(k_1,\ldots,k_r;p+1)$ can be expressed as a finite sum of MTVs.

For a composition $\bfk=(k_1,\ldots,k_r)\in \N^r$, put $\bfk_{+}=(k_1,\ldots,k_{r-1},k_r+1)$. Recall that the usual dual of an admissible composition $\bfk$, denoted by $\bfk^*$, is defined as follows:
for all $a_1,\dots,a_n,b_1,\dots,b_n\in\N$,
\begin{equation*}
\big(\{1\}_{a_1-1},b_1+1,\ldots,\{1\}_{a_n-1},b_n+1\big)^*=
\big(\{1\}_{b_n-1},a_n+1,\ldots,\{1\}_{b_1-1},a_1+1\big)
\end{equation*}
For $\bfj=(j_1,\ldots,j_r)\in \N^r_0\ (\N_0:=\N\cup \{0\})$, we set $|\bfj|=j_1+\cdots+j_r$ and call it the weight of $\bfj$, and $\dep(\bfj)=r$, the depth of $\bfj$.  For two such compositions $\bfj$  and $\bfk$ of the same depth, we denote by $\bfj+\bfk$ the composition obtained by the component-wise addition, $\bfj+\bfk=(j_1+k_1,\ldots,j_r+k_r)$, and by $b(\bfk;\bfj)$ the quantity given by
\[b(\bfk;\bfj):=\prod\limits_{i=1}^r \binom{k_i+j_i-1}{j_i}.\]
In \cite{KT2018}, Kaneko and Tsumura obtained explicit expressions in terms of multiple zeta/zeta-star values for $\xi(\bfk;p+1)$ and $\eta(\bfk;p+1)$ as follows.
\begin{thm} (\cite{KT2018} Theorem 2.5) For any composition $\bfk=(k_1,\ldots,k_r)$ and any $p\in\N\setminus \{1\}$, we have
\begin{equation}\label{4.6}
\xi(\bfk;p+1)=\sum_{|\bfj|=p,\ \dep(\bfj)=n} b\left((\bfk_{+})^*;\bfj\right)\zeta\left((\bfk_{+})^*+\bfj\right)
\end{equation}
and
\begin{equation}\label{4.7}
\eta(\bfk;p+1)=(-1)^{r-1}\sum_{|\bfj|=p,\ \dep(\bfj)=n} b\left((\bfk_{+})^*;\bfj\right)\zeta^\star\left((\bfk_{+})^*+\bfj\right),
\end{equation}
where both sums are over all $\bfj\in \N^r_0$ of weight $p$ and depth $n:=\dep(\bfk^*_{+})=|\bfk|+1-\dep(\bfk).$
\end{thm}
Previously, Kawasaki and Ohno \cite{KO2018} gave alternative proofs of \eqref{4.6} and \eqref{4.7} by the associated integral representations of $\xi(\bfk;p+1)$ and $\eta(\bfk;p+1)$ (see  \cite[Theorem 2.2]{KO2018}). By a similar argument as in the proof of Theorem 2.2 in \cite{KO2018}, we can prove the following theorem.
\begin{thm} \label{thmc3} For any composition $\bfk=(k_1,\ldots,k_r)$ and any $p\in\N\setminus \{1\}$, we have
\begin{equation}\label{4.}
\psi(\bfk;p+1)=\sum_{|\bfj|=p,\ \dep(\bfj)=n} b\left((\bfk_{+})^*;\bfj\right)T\left((\bfk_{+})^*+\bfj\right).
\end{equation}
where both sums are over all $\bfj\in \N^r_0$ of weight $p$ and depth $n:=\dep(\bfk^*_{+})=|\bfk|+1-\dep(\bfk).$
\end{thm}
\begin{proof}
The proof is completely similar as the proof of \cite[Theorem 2.2]{KO2018}) and is thus omitted.
\end{proof}

It is clear that formula \eqref{b15} is an immediate corollary of Theorem \ref{thmc3}. Theorems \ref{thmb3} and \ref{thmc3} now yield the following corollary.
\begin{cor}\label{cor:convolutedT}
For positive integer $p$ and composition $\bfk$, the convoluted $T$-value $T(\bfk\circledast\{1\}_p)$ can be expressed as a linear combination of products of {\rm MTVs} and alternating single zeta values with $\Z$-coefficients.
\end{cor}

We record several examples to illustrate Corollary \ref{cor:convolutedT} .
\begin{align*}
T(\bfk_{2m}\circledast \{1\}_{3})&\,=\psi(\bfk_{2m};3)-\log(2)\psi(\bfk_{2m};2),\\
T(\bfk_{2m-1}\circledast \{1\}_{3})&\,=\psi(\bfk_{2m-1};3)-2{\bar \zeta}(2)T((\bfk_{2m-1})_+),\\
T(\bfk_{2m}\circledast \{1\}_{4})&\,=\psi(\bfk_{2m};4)-{\bar \zeta}(2)\psi(\bfk_{2m};2),\\
T(\bfk_{2m-1}\circledast \{1\}_{4})&\,=\psi(\bfk_{2m-1};4)-2{\bar \zeta}(3)T((\bfk_{2m-1})_+)\\
& -2\log(2)\psi(\bfk_{2m-1};3)+4\log(2){\bar \zeta}(2)T((\bfk_{2m-1})_+).
\end{align*}

\begin{thm} For any positive integers $l_1,l_2$ and composition $\bfk_{m}\in \N^m$,
\begin{align*}
T(\bfk_{m}\circledast (l_1,l_2))=f(\bfk_m,l_1,l_2)+
I\left(
\raisebox{12pt}{\begin{xy}
{(-3,-18) \ar @{{*}-} (0,-15)},
{(0,-15) \ar @{{o}.} (3,-12)},
{(3,-12) \ar @{{o}.} (9,-6)},
{(9,-6) \ar @{{*}-} (12,-3)},
{(12,-3) \ar @{{o}.} (15,0)},
{(15,0) \ar @{{o}-} (18,3)},
{(18,3) \ar @{{o}-} (21,6)},
{(21,6) \ar @{{o}.} (24,9)},
{(24,9) \ar @{{o}-} (27,3)},
{(27,3) \ar @{{*}-} (30,6)},
{(30,6) \ar @{{o}.} (33,9)},
{(33,9) \ar @{{o}-} },
{(-3,-17) \ar @/^1mm/ @{-}^{k_1} (2,-12)},
{(9,-5) \ar @/^1mm/ @{-}^{k_{m}} (14,0)},
{(18,4) \ar @/^1mm/ @{-}^{l_2} (23,9)},
{(28,3) \ar @/_1mm/ @{-}_{l_{1}} (33,8)}.
\end{xy}}
\right)
\end{align*}
where $f(\bfk_m,l_1,l_2)=0$ if $m$ is even and $f(\bfk_m,l_1,l_2)=2\bar\zeta(l_1)T(k_1,\dots,k_{m-1},k_{m}+l_2)$ if $m$ is odd.
\end{thm}
\begin{proof} Let $j\in\N$. Then
\begin{equation*}
I:=I\left(
\raisebox{12pt}{\begin{xy}
{(-3,-18) \ar @{{*}-} (0,-15)},
{(0,-15) \ar @{{o}.} (3,-12)},
{(3,-12) \ar @{{o}.} (9,-6)},
{(9,-6) \ar @{{*}-} (12,-3)},
{(12,-3) \ar @{{o}.} (15,0)},
{(15,0) \ar @{{o}-} (18,3)},
{(18,3) \ar @{{o}-} (21,6)},
{(21,6) \ar @{{o}.} (24,9)},
{(24,9) \ar @{{o}-} (27,3)},
{(27,3) \ar @{{*}-} (30,6)},
{(30,6) \ar @{{o}.} (33,9)},
{(33,9) \ar @{{o}-} },
{(-3,-17) \ar @/^1mm/ @{-}^{k_1} (2,-12)},
{(9,-5) \ar @/^1mm/ @{-}^{k_{j}} (14,0)},
{(18,4) \ar @/^1mm/ @{-}^{l_2} (23,9)},
{(28,3) \ar @/_1mm/ @{-}_{l_{1}} (33,8)}.
\end{xy}}
\right)
=\int_0^1   f(x)  \left(\int_0^x \omz^{l_2-1} \omz^{k_j-1}\omn \cdots \omz^{k_1-1}\omn \right) \frac{dx}{x}
\end{equation*}
where
\begin{equation*}
    f(x)= \int_{0<t_1<x,\  t_1<t_2<\dots<t_{l_1}<1} \omz(t_{l_1})\cdots  \omz(t_2)  \omn(t_1)=\int_0^x \frac{\log^{l_1-1}(t)}{(l_1-1)!}\frac{2dt}{1-t^2}.
\end{equation*}
Easy computation yields that
\begin{align*}
&\, \int_0^x \omz^{l_2-1} \omz^{k_j-1}\omn \cdots \omz^{k_1-1}\omn \\
=&\, \sum_{0<n_1<\dots<n_j<n} \frac{ (1-(-1)^{n_1} )(1+(-1)^{n_2} ) \cdots (1-(-1)^{j+n_{j-1}} ) }{n_1^{k_1}n_2^{k_2} \cdots n_{j-1}^{k_{j-1}} } \frac{ x^n(1+(-1)^{j+n})}{n^{k_j+l_2-1} }\\
=&\, \left\{
       \begin{array}{ll}
        \sum_{n>0} \frac{2T_n(k_1,\dots,k_{2m-1}) }{(2n)^{k_{2m}+l_2} }(2n) x^{2n}, & \hbox{if $j=2m$;} \\
         \sum_{n>0} \frac{2T_n(k_1,\dots,k_{2m-2})}{(2n-1)^{k_{2m-1}+l_2} }(2n-1) x^{2n-1}, \qquad \ & \hbox{if $j=2m-1$.}
       \end{array}
     \right.
\end{align*}
Assume $l_1>1$ first. Then using integration by parts we get
\begin{align*}
  \int_0^1  r x^r f(x) \frac{dx}{x} =&\,  \int_0^1 f(x)\, d(x^r)=f(1) -\int_0^1 x^r f'(x) \, dx \\
=&\, T(l_1)-\sum_{m>r} \frac{(1+(-1)^{m+r})}{m^{l_1}}
=
\left\{
       \begin{array}{ll}
       T_n(l_1) , & \hbox{if $r=2n$;} \\
       2\bar\zeta(l_1)+S_n(l_1) , \qquad  \ & \hbox{if $r=2n-1$.}
       \end{array}
     \right.
\end{align*}
One checks easily this still holds if $l_1=1$. Indeed, we may compute $I$ by the following method. We have
\begin{equation*}
I=\int_0^1 \left(\int_t^1  \left(\int_0^x \omz^{l_2-1} \omz^{k_j-1}\omn \cdots \omz^{k_1-1}\omn \right) \frac{dx}{x}\right)\frac{2dt}{1-t^2}.
\end{equation*}
By the same computation as in the case $l_1>1$, we only need to compute
\begin{equation*}
A_r=\int_0^1 \left(\int_t^1 r x^r \frac{dx}{x}\right)\frac{2dt}{1-t^2}= \int_0^1 \frac{2(1-t^r)}{1-t^2} \,dt.
\end{equation*}
If $r=2n$, by geometric series we see immediately that
 \begin{equation*}
 A_{2n}= \sum_{i=1}^{n} \frac{2}{2i-1}=T_n(1).
 \end{equation*}
If $r=2n-1$ then
\begin{equation*}
 A_{2n-1}= \int_0^1  \frac{2(1-t)}{1-t^2}\, dt + \int_0^1  \frac{2(t-t^{2n-1})}{1-t^2}\,dt
=2\log(2)+ \sum_{i=1}^{n-1} 2\int_0^1  t^{2i-1} \, dt=\log(2)+ S_n(1).
\end{equation*}
We have completed the proof of the theorem.
\end{proof}

\begin{re} The above theorem does not seem to generalize to arbitrary $T(\bfk\circledast\bfl)$. In the database associated with the paper \cite{BlumleinBrVe2010}, Blumlein et al. constructed explicit conjectural basis for the $\Q$-vector space generated by Euler sums up to weight 12. In particular, there are 89 basis elements for the weight 10 piece since 89 is the conjectural dimension of that subspace.  Using their database we find by Maple that both  $T((2, 1, 1)\circledast (1, 1, 3, 1))$
and $T( (2, 1, 1)\circledast (1, 2, 2, 1))$ need the basis element $\zeta(\bar 5)\zeta(1,1,\bar 3)$. However, none of the
products of MTVs with single alternating MZVs involves this elements.
\end{re}

\section{Duality for Kaneko-Tsumura $\psi$-Values}\label{sec:psi-V}
In this section we give explicit expressions of $T(\bfk\circledast\{1\}_p)$ by Kaneko-Tsumura $\psi$-values and alternating zeta values, and find some duality relations for Kaneko-Tsumura $\psi$-Values.  We need the following lemma.
\begin{lem}\label{lemd1}  Let $A_{p,q}, B_p, C_p\ (p,q\in \N)$ be any complex sequences. If
\begin{align}\label{e1}
\sum\limits_{j=1}^p A_{j,p}B_j=C_p,\quad A_{p,p}:=1,
\end{align}
holds, then
\begin{align}\label{e2}
B_p=\sum\limits_{j=1}^p C_j \sum\limits_{k=1}^{p-j} (-1)^k \left\{\sum\limits_{i_0<i_1<\cdots<i_{k-1}<i_k,\atop i_0=j,i_k=p} \prod\limits_{l=1}^k A_{i_{l-1},i_l}\right\},
\end{align}
where $\sum\limits_{k=1}^0 (\cdot):=1$.
\end{lem}
\begin{proof}
 We proceed with induction on $p$. For $p=1$ we have $B_1=C_1$, and the formula is true. For $p>1$ we proceed as follows. By \eqref{e1},
\begin{align*}
B_{p+1}=C_{p+1}-\sum\limits_{j=1}^p (-1)^{j+1} A_{j,p+1}B_j.
\end{align*}
Then by the induction hypothesis, we see that
\begin{align*}
B_{p+1}&=C_{p+1}-\sum\limits_{l=1}^p A_{l,p+1} \sum\limits_{j=1}^l C_j \sum\limits_{k=1}^{l-j} (-1)^k \left\{\sum\limits_{i_0<i_1<\cdots<i_{k-1}<i_k,\atop i_0=j,i_k=l} \prod\limits_{m=1}^k A_{i_{m-1},i_m}\right\}\\
&=C_{p+1}+\sum\limits_{j=1}^p C_j \sum\limits_{l=j}^{p}A_{l,p+1}\sum\limits_{k=1}^{l-j} (-1)^{k+1} \left\{\sum\limits_{i_0<i_1<\cdots<i_{k-1}<i_k,\atop i_0=j,i_k=l} \prod\limits_{m=1}^k A_{i_{m-1},i_m}\right\}\\
&=C_{p+1}+\sum\limits_{j=1}^p C_j \sum\limits_{k=1}^{p-j} (-1)^{k+1} \left\{\sum\limits_{i_0<i_1<\cdots<i_{k-1}<i_k<i_{k+1},\atop i_0=j,i_{k+1}=p+1} \prod\limits_{m=1}^{k+1} A_{i_{m-1},i_m}\right\}- \sum\limits_{j=1}^p C_j A_{j,p+1}\\
&=C_{p+1}+\sum\limits_{j=1}^p C_j \sum\limits_{k=2}^{p+1-j} (-1)^{k} \left\{\sum\limits_{i_0<i_1<\cdots<i_{k-1}<i_k,\atop i_0=j,i_{k}=p+1} \prod\limits_{m=1}^{k} A_{i_{m-1},i_m}\right\}- \sum\limits_{j=1}^p C_j A_{j,p+1}\\
&=C_{p+1}+\sum\limits_{j=1}^p C_j \sum\limits_{k=1}^{p+1-j} (-1)^{k} \left\{\sum\limits_{i_0<i_1<\cdots<i_{k-1}<i_k,\atop i_0=j,i_{k}=p+1} \prod\limits_{m=1}^{k} A_{i_{m-1},i_m}\right\}\\
&=\sum\limits_{j=1}^{p+1} C_j \sum\limits_{k=1}^{p+1-j} (-1)^{k} \left\{\sum\limits_{i_0<i_1<\cdots<i_{k-1}<i_k,\atop i_0=j,i_{k}=p+1} \prod\limits_{m=1}^{k} A_{i_{m-1},i_m}\right\}.
\end{align*}
Thus, the formula \eqref{e2} holds.\end{proof}

Using Lemma \ref{lemd1}, we obtain the explicit formulas of $T(\bfk\circledast\{1\}_p)$ via Kaneko-Tsumura $\psi$-values and alternating zeta values. For positive integers $j$ and $p$ with $1\leq j\leq p$, set
\begin{align}
Z(j,p):=\sum_{k=1}^{p-j}(-2)^k\sum\limits_{i_0<i_1<\cdots<i_{k-1}<i_k,\atop i_0=j,i_k=p} \prod\limits_{l=1}^k {\bar \zeta}(2i_l-2i_{l-1}),\quad Z(p,p):=1.
\end{align}
It is clear that $Z(j,p)=a_{j,p}\zeta(2p-2j)\quad (a_{j,p}\in \mathbb{Q})$. In fact, we have the following explicit evaluation.
\begin{pro} For any $j\in\N$ and $w\in\N_0$,
\begin{align}
Z(j,j+w)=\frac{(-1)^{w}\pi^{2w}}{(2w+1)!}.
\end{align}
\end{pro}
\begin{proof} An elementary calculation yields that
\begin{align*}
Z(j,j+w)=&\sum_{k=1}^{w}(-2)^k\sum\limits_{i_0<i_1<\cdots<i_{k-1}<i_k,\atop i_0=j,i_k=j+w} \prod\limits_{l=1}^k {\bar \zeta}(2i_l-2i_{l-1})\\
=&\sum_{k=1}^w 2^k \sum_{i_1+\cdots+i_k=w,\atop \forall i_l\geq 1} \zeta(\overline{2i_1})\zeta(\overline{2i_2})\cdots\zeta(\overline{2i_k}).
\end{align*}
Hence, it suffices to show that
\begin{equation*}
1+\sum_{w\ge 1} \left(\sum_{k=1}^w 2^k \sum_{i_1+\cdots+i_k=w} \zeta(\ol{2i_1})\cdots \zeta(\ol{2i_k}) \right) x^{2w}= 1+\sum_{w\ge 1} \frac{(-1)^w \pi^{2w}}{(2w+1)!} x^{2w}=\frac{\sin(\pi x)}{\pi x}.
\end{equation*}
By change of summation orders, we get
\begin{align*}
&\, \sum_{w\ge 1} \left(\sum_{k=1}^w 2^k \sum_{i_1+\cdots+i_k=w} \zeta(\ol{2i_1})\cdots \zeta(\ol{2i_k}) \right) x^{2w}\\
=&\,\sum_{k\ge 1} \sum_{w\ge k} \ \sum_{i_1+\cdots+i_k=w} \ \prod_{l=1}^k 2 \zeta(\ol{2i_l}) x^{2i_l}\\
=&\,\sum_{k\ge 1}\  \sum_{i_1+\cdots+i_k\ge k} \ \prod_{l=1}^k \left( \sum_{n_l\ge 1} \frac{2 (-1)^{n_l} x^{2i_l} }{n_l^{2i_l}} \right)\\
=&\,\sum_{k\ge 1}   \left( \sum_{i\ge 1} \sum_{n\ge 1} \frac{2 (-1)^{n} x^{2i} }{n^{2i}} \right)^k\\
=&\,\sum_{k\ge 1}   \left(  \sum_{n\ge 1} \frac{2 (-1)^{n} x^{2}/n^2 }{1-x^2/n^{2}} \right)^k\\
=&\, \frac{g(x)}{1-g(x)}
\end{align*}
where
\begin{equation*}
g(x)= \sum_{n\ge 1} \frac{2 (-1)^{n} x^{2}/n^2 }{1-x^2/n^{2}}
\end{equation*}
Notice that
\begin{equation*}
 f(x):= \sum_{n\ge 1} \frac{2x^{2}/n^2 }{1-x^2/n^{2}} =-x \frac{d}{dx} \left(\log \prod_{n\ge 1} \Big(1-\frac{x^2}{n^2}\Big)\right)
= -x \frac{d}{dx} \left( \log \frac{\sin(\pi x)}{\pi x}\right)=1-\pi x \cot(\pi x).
\end{equation*}
So we have
\begin{equation*}
g(x)=f\Big(\frac x2\Big)-h(x)    \quad \text{where } h(x)= \sum_{n \text{ odd}} \frac{2 x^{2}/n^2 }{1-x^2/n^{2}}.
\end{equation*}
But it is clear that
\begin{equation*}
f\Big(\frac x2\Big)+h(x)=f(x).
\end{equation*}
Thus
\begin{align*}
    g(x)=&\, f\Big(\frac x2\Big)-h(x)=2f\Big(\frac x2\Big)-f(x)\\
=&\, 1-\pi x \cot \Big(\frac\pi 2\Big)+\pi x \cot (\pi x)=1-\pi x \csc(\pi x).
\end{align*}
Here we have used the identity $\cot(\theta)-\cot(2\theta)=\csc(2\theta)$.
Finally,
\begin{align*}
  1+  \frac{g(x)}{1-g(x)} =\frac{1}{1-g(x)}= \frac{\sin(\pi x)}{\pi x}
\end{align*}
as desired.
\end{proof}

\begin{thm} For any composition $\bfk_{2m-1}\in\N^{2m-1}$ and $p\in\N$
\begin{align}\label{d4}
T(\bfk_{2m-1}\circledast \{1\}_{2p+1})&=\sum_{j=1}^p \left(\psi(\bfk_{2m-1};2j+1)-2{\bar \zeta}(2j)T((\bfk_{2m-1})_+) \right)Z(j,p).
\end{align}
\end{thm}
\begin{proof} In Lemma \ref{lemd1}, setting $A_{j,p}:=2{\bar \zeta}(2p-2j)$, $B_j=T(\bfk_{2m-1}\circledast\{1\}_{2j+1})$ and $C_p=\psi(\bfk_{2m-1};2p+1)-2{\bar \zeta}(2p)T((\bfk_{2m-1})_+)$ and using \eqref{b7}, we obtain the desired evaluation.\end{proof}

\begin{thm} For any composition $\bfk_{2m}\in\N^{2m}$ and $p\in\N$,
\begin{align}\label{d5}
T(\bfk_{2m}\circledast \{1\}_{2p})&=\sum_{j=1}^p \psi(\bfk_{2m};2j)Z(j,p).
\end{align}
\end{thm}
\begin{proof} In Lemma \ref{lemd1}, setting $A_{j,p}:=2{\bar \zeta}(2p-2j)$, $B_j=T(\bfk_{2m}\circledast\{1\}_{2j})$ and $C_p=\psi(\bfk_{2m};2p)$ and using \eqref{b8}, we obtain the desired evaluation.\end{proof}

Clearly, applying \eqref{d4} and \eqref{d5} to \eqref{b6} and \eqref{b9}, we can evaluate the explicit evaluations of the convoluted $T$-values $T(\bfk_{2m-1}\circledast \{1\}_{2p})$ and $T(\bfk_{2m}\circledast \{1\}_{2p+1})$ in terms of Kaneko-Tsumura $\psi$-values and (alternating) single zeta values. Next, we prove a few duality relations of Kaneko-Tsumura $\psi$-Values.

\begin{thm}\label{thmd4} For positive integers $k,m$ and $p$,
\begin{align}
&\sum_{j=1}^p \left(\psi(\{1\}_{2m},k;2j+1)-2{\bar \zeta}(2j)T(\{1\}_{2m},k+1) \right)Z(j,p)\nonumber\\
&=\sum_{j=1}^m \left(\psi(\{1\}_{2p},k;2j+1)-2{\bar \zeta}(2j)T(\{1\}_{2p},k+1) \right)Z(j,m).
\end{align}
\end{thm}
\begin{proof} In \eqref{d4}, replacing $m$ by $m+1$, then setting $k_1=\cdots=k_{2m}=1$ and $k_{2m+1}=k$ we find that
\begin{align*}
&T((\{1\}_{2m},k)\circledast \{1\}_{2p+1})=\su \frac{T_n(\{1\}_{2m})T_n(\{1\}_{2p})}{(2n-1)^{k+1}}\\
&=\sum_{j=1}^p \left(\psi(\{1\}_{2m},k;2j+1)-2{\bar \zeta}(2j)T(\{1\}_{2m},k+1) \right)Z(j,p).
\end{align*}
Thus, changing $(m,p)$ to $(p,m)$ we obtain
\begin{align*}
&T((\{1\}_{2m},k)\circledast \{1\}_{2p+1})=\sum_{j=1}^m \left(\psi(\{1\}_{2p},k;2j+1)-2{\bar \zeta}(2j)T(\{1\}_{2p},k+1) \right)Z(j,m).
\end{align*}
Hence, we complete this proof.\end{proof}

\begin{thm} For positive integers $k,m$ and $p$,
\begin{align}
&\sum_{j=1}^p \psi(\{1\}_{2m-1},k;2j)Z(j,p)=\sum_{j=1}^m \psi(\{1\}_{2p-1},k;2j)Z(j,m).
\end{align}
\end{thm}
\begin{proof} By a similar argument as in the proof of Theorem \ref{thmd4}, letting $k_1=\cdots=k_{2m-1}=1$ and $k_{2m}=k$ in \eqref{d5}, by a straightforward calculation, we obtain the desired evaluation.\end{proof}

Similarly, for positive integers $j$ and $p$ with $1\leq j\leq p$, let
\begin{align}
\widetilde{Z}(j,p):=\sum_{k=1}^{p-j}\frac{(-1)^k}{2\log^{k+1}(2)}\sum\limits_{i_0<i_1<\cdots<i_{k-1}<i_k,\atop i_0=j,i_k=p} \prod\limits_{l=1}^k {\bar \zeta}(2i_l-2i_{l-1}+1),\quad \widetilde{Z}(p,p):=\frac{1}{2\log(2)}.
\end{align}
Using \eqref{b6}, \eqref{b9} and Lemma \ref{lemd1}, by elementary calculations, we can quickly deduce the next two results:
\begin{align}\label{d9}
&T(\bfk_{2m-1}\circledast \{1\}_{2p-1})=\sum_{j=1}^p \left(\psi(\bfk_{2m-1};2j)-T(\bfk_{2m-1}\circledast \{1\}_{2j}) \right)\widetilde{Z}(j,p)
\end{align}
and
\begin{align}\label{d10}
&T(\bfk_{2m}\circledast \{1\}_{2p})=\sum_{j=1}^p \left(\psi(\bfk_{2m};2j+1)-T(\bfk_{2m}\circledast \{1\}_{2j+1}) \right)\widetilde{Z}(j,p).
\end{align}
Setting $k_1=\cdots=k_{2m-2}=1$ and $k_{2m-1}=k$ in \eqref{d9}, and $k_1=\cdots=k_{2m-1}=1$ and $k_{2m}=k$ in \eqref{d10}, we readily arrive at the following two duality relations
\begin{multline}  \label{d11}
 \sum_{j=1}^p \left(\psi(\{1\}_{2m-2},k;2j)-T((\{1\}_{2m-2},k)\circledast \{1\}_{2j}) \right)\widetilde{Z}(j,p) \\
= \sum_{j=1}^m \left(\psi(\{1\}_{2p-2},k;2j)-T((\{1\}_{2p-2},k)\circledast \{1\}_{2j}) \right)\widetilde{Z}(j,m),
\end{multline}
and
\begin{multline} \label{d12}
 \sum_{j=1}^p \left(\psi(\{1\}_{2m-1},k;2j+1)-T((\{1\}_{2m-1},k)\circledast \{1\}_{2j+1}) \right)\widetilde{Z}(j,p) \\
= \sum_{j=1}^m \left(\psi(\{1\}_{2p-1},k;2j+1)-T((\{1\}_{2p-1},k)\circledast \{1\}_{2j+1}) \right)\widetilde{Z}(j,m).
\end{multline}

\section{Explicit Evaluations of Multiple $S$-Values}\label{sec:MSV}
In this section we will use the method of contour integration and residue theorem to evaluate MSVs at depth two and three. We define a complex kernel function $\xi(s)$ with two requirements: (i). $\xi(s)$ is meromorphic in the whole complex plane. (ii). $\xi(s)$ satisfies $\xi (s)=o(s)$ over an infinite collection of circles $|s|=\rho_k$ with $\rho_k\to \infty $. Applying these two conditions of kernel function $\xi(s)$, Flajolet and Salvy discovered the following residue lemma.
\begin{lem}\emph{(\cite{FS1998})}\label{lemf1}
Let $\xi(s)$ be a kernel function and let $r(s)$ be a rational function which is $O(s^{-2})$ at infinity. Then
\begin{align}
\sum_{\alpha\in O} \Res(r(s)\xi(s),\alpha)+ \sum_{\beta\in S} \Res(r(s)\xi(s),\beta) = 0.
\end{align}
where $S$ is the set of poles of $r(s)$ and $O$ is the set of poles of $\xi(s)$ that are not poles $r(s)$ . Here $\Res(r(s),\alpha)$ denotes the residue of $r(s)$ at $s=\alpha$.
\end{lem}

For convenience, set
\begin{align*}
\t(k_1,k_2,\ldots,k_r)&\,:=2^{k_1+k_2+\cdots+k_r} t(k_1,k_2,\ldots,k_r),\\
\S(k_1,k_2,\ldots,k_r)&\,:=2^{k_1+k_2+\cdots+k_r-r} S(k_1,k_2,\ldots,k_r).
\end{align*}

\begin{thm} \label{thm:evalMSV}
For an odd weight $p+q$, the double sums $\S(p,q)$ (or $\S(p,q)$) are reducible to zeta values,
\begin{align}
(1-(-1)^{p+q})\S(p,q)&=2(-1)^p \sum_{k=0}^{[p/2]} \binom{p+q-2k-1}{q-1} \zeta(2k)\t(p+q-2k)\nonumber\\
&\quad+(-1)^p \sum_{k=0}^{q-1} (1-(-1)^k) \binom{p+q-k-2}{p-1}\t(k+1)\t(p+q-k-1)\nonumber\\
&\quad-(-1)^p(1+(-1)^q)\zeta(p)\t(q),
\end{align}
where $\zeta(0),\ \zeta(1)$ and $\t(1)$ should be interpreted as $-1/2,\ 0$ and $2\log(2)$, respectively.
\end{thm}
\begin{proof} We consider the contour integral
\begin{align*}
\oint\limits_{\left( \infty  \right)}F(s)ds=\oint\limits_{\left( \infty  \right)} \frac{\pi \cot(\pi s)\psi^{(p-1)}(-s)}{(s+1/2)^q (p-1)!} ds\quad(q\geq 2,q\geq 1),
\end{align*}
where $\oint_{\left( \infty  \right)}$ denotes integration along large circles, that is, the limit of integrals $\oint_{\left| s \right| = R}$. If $p=1$, replace $\psi(-s)$ by $\psi(-s)+\gamma$. Clearly, ${\pi \cot(\pi s)\psi^{(p-1)}(-s)}/(p-1)!$ is a kernel function. Hence, $\oint_{\left( \infty  \right)}F(s)ds=0$.
The function $F(s)$ only has poles at the $n$ and $-1/2$, $n$ is any integer. By residue theorem, we deduce
\begin{align*}
&\Res(F(s),-n)=\frac{\zeta(p)-H_{n-1}^{(p)}}{(n-1/2)^q}(-1)^{p+q},\quad n\geq 1,\\
&\Res(F(s),n)=(-1)^p\binom{p+q-1}{p}\frac{1}{(n+1/2)^{p+q}} +\frac{(-1)^p\zeta(p)+H_n^{(p)}}{(n+1/2)^q}\\&\quad\quad\quad\quad\quad\quad\quad-2(-1)^p \sum_{k=1}^{[p/2]}\binom{p+q-2k-1}{q-1}\frac{\zeta(2k)}{(n+1/2)^{p+q-2k}},\quad n\geq 0,\\
&\Res(F(s),-1/2)=-(-1)^p \sum_{k=0}^{q-1} (1-(-1)^k) \binom{p+q-k-2}{p-1}\t(k+1)\t(p+q-k-1).
\end{align*}
Thus, applying Lemma \ref{lemf1} and summing these three contributions we can quickly deduce the statement of the theorem.
\end{proof}

As two simple examples of Theorem~\ref{thm:evalMSV}, we have
\begin{align*}
\S(3,2)=62\zeta(5)-32\zeta(2)\zeta(3),\quad S(3,2)=\frac{31}{4}\zeta(5)-4\zeta(2)\zeta(3).
\end{align*}
The next result provides some relations among double $S$-values.
\begin{thm} For positive integers $m,p$ and $q>1$,
\begin{align}
&(-1)^{m-1} \sum\limits_{i+j=p-1,\atop i,j\geq 0} \binom{m+i-1}{i}\binom{q+j-1}{j}\S(m+i,q+j) \nonumber\\
+&(-1)^{p-1} \sum\limits_{i+j=m-1,\atop i,j\geq 0} \binom{p+i-1}{i}\binom{q+j-1}{j} \S(p+i,q+j)\nonumber\\
&=\binom{p+q+m-2}{q-1}\t(p+q+m-1) \nonumber\\&\quad+ \sum\limits_{i+j=p-1,\atop i,j\geq 0} \binom{m+i-1}{i}\binom{q+j-1}{j} (-1)^i \zeta(m+i)\t(q+j)\nonumber\\&\quad+\sum\limits_{i+j=m-1,\atop i,j\geq 0} \binom{p+i-1}{i}\binom{q+j-1}{j} (-1)^i \zeta(p+i)\t(q+j)\nonumber\\
&\quad-\sum\limits_{i+j=q-1,\atop i,j\geq 0} \binom{m+i-1}{i}\binom{p+j-1}{j} \t(m+i)\t(p+j),
\end{align}
where $\zeta \left(1\right)$ and $\t(1)$ should be interpreted as $0$ and $-2\log(2)$ wherever it occurs.
\end{thm}
\begin{proof} Consider the contour integral
\begin{align*}
\oint\limits_{\left( \infty  \right) } \frac{\psi^{(m-1)}(-s)\psi^{(p-1)}(-s)}{(s+1/2)^q(m-1)!(p-1)!}ds=0
\end{align*}
and use residue computations to obtain the desired evaluation.\end{proof}

Next, we evaluate the triple $S$-values. According to the definition of triple $S$-values, we have
\begin{equation}\label{equ:TripleS}
 \S(k_1,k_2,k_3)=\zeta(k_3)\S(k_1,k_2)-\su \frac{H^{(k_1)}_{n-1}H^{(k_3)}_{n-1}}{(n-1/2)^{k_2}},\quad (k_1\geq 1,k_2,k_3\geq 2).
\end{equation}
By using the method of contour integration and residue theorem, we get the following theorem.
\begin{thm}\label{thm:ParityTripleS} For positive integers $m,p$ and $q>1$,
\begin{align}
&(1+(-1)^{p+q+m})\su \frac{H^{(m)}_{n-1}H^{(p)}_{n-1}}{(n-1/2)^q}\nonumber\\
&=(-1)^{p+q+m} \left(\zeta(m)\S(p,q)+\zeta(p)\S(m,q) \right)-(-1)^{p+m} (1+(-1)^q)\zeta(m)\zeta(p)\t(q)\nonumber\\
&\quad-(-1)^m\zeta(m)\S(p,q)-(-1)^p\S(m,q)-(-1)^{p+m}\binom{p+q+m-1}{q-1}\t(p+q+m)\nonumber\\
&\quad+2(-1)^{p+m} \sum_{k=1}^{[(p+m)/2]} \binom{p+q+m-2k-1}{q-1}\zeta(2k)\t(p+q+m-2k)\nonumber\\
&\quad+(-1)^{p+m} \sum_{k=1}^{p+1} (-1)^k \binom{k+m-2}{m-1} \binom{p+q-k}{q-1} \left\{\begin{array}{l} \zeta(k+m-1)\t(p+q-k+1)\\ -(-1)^{k+m}\S(k+m-1,p+q-k+1)  \end{array}\right\} \nonumber\\
&\quad+(-1)^{p+m} \sum_{k=1}^{p+1} (-1)^k \binom{k+p-2}{p-1} \binom{m+q-k}{q-1} \left\{\begin{array}{l} \zeta(k+p-1)\t(m+q-k+1)\\ -(-1)^{k+p}\S(k+p-1,m+q-k+1)  \end{array}\right\} \nonumber\\
&\quad+2(-1)^{p+m} \sum_{2k_1+k_2\leq p+1,\atop k_1,k_2\geq 1} \binom{k_2+m-2}{m-1} \binom{p+q-2k_1-k_2}{q-1} \zeta(2k_1)\nonumber\\ &\quad\quad\quad\quad\quad\quad\quad\quad\quad\quad\times\left\{\begin{array}{l} \zeta(k_2+m-1)\t(p+q-2k_1-k_2+1)\\ -(-1)^{k_2+m}\S(k_2+m-1,p+q-2k_1-k_2+1)  \end{array}\right\} \nonumber\\
&\quad+2(-1)^{p+m} \sum_{2k_1+k_2\leq m+1,\atop k_1,k_2\geq 1} \binom{k_2+p-2}{p-1} \binom{m+q-2k_1-k_2}{q-1} \zeta(2k_1)\nonumber\\ &\quad\quad\quad\quad\quad\quad\quad\quad\quad\quad\times\left\{\begin{array}{l} \zeta(k_2+p-1)\t(m+q-2k_1-k_2+1)\\ -(-1)^{k_2+p}\S(k_2+p-1,m+q-2k_1-k_2+1)  \end{array}\right\} \nonumber\\
&\quad+(-1)^{p+m} \sum_{k_1+k_2+k_3=q-1,\atop k_1,k_2,k_3\geq 0} (1-(-1)^{k_1})\binom{k_2+m-1}{m-1}\binom{k_3+p-1}{p-1} \t(k_1+1)\t(k_2+m)\t(k_3+p),
\end{align}
where $\zeta \left(1\right)$ and $\t(1)$ should be interpreted as $0$ and $2\log(2)$ wherever it occurs.
\end{thm}
\begin{proof} Considering the contour integral
\begin{align*}
\oint\limits_{\left( \infty  \right) } \pi \cot(\pi s)\frac{\psi^{(m-1)}(-s)\psi^{(p-1)}(-s)}{(s+1/2)^q(m-1)!(p-1)!}ds=0
\end{align*}
and applying residue theorem, we may deduce the desired evaluation after a rather tedious computation.\end{proof}

\begin{cor}\label{cor:ParityTripleS}
For positive integer $q>1$,
\begin{align}
(1+(-1)^q) \su \frac{H^2_{n-1}}{(n-1/2)^q}&=2\S(q,2)+2q\S(q+1,1)+4\zeta(2)\t(q)-\frac{q(q+1)}{2}\t(q+2)\nonumber\\&\quad+\sum_{k_1+k_2+k_3=q-1,\atop k_1,k_2,k_3\geq 0} (1-(-1)^{k_1}) \t(k_1+1)\t(k_2+1)\t(k_3+1).
\end{align}
\end{cor}
\begin{proof} This follows immediately from Theorem \ref{thm:ParityTripleS} by setting $m =p=1$.
\end{proof}

Theorem \ref{thm:ParityTripleS} together with \eqref{equ:TripleS}
provides an explicit form of the following parity result.
\begin{cor}
For positive integers $k_1\geq 1$ and $k_2,k_3\geq 2$, if weight $k_1+k_2+k_3$ is an even, then the triple $S$-values $S(k_1,k_2,k_3)$ can be expressed as a rational linear combination of products of double $S$-values and Riemann zeta values.
\end{cor}

More general, we have the following parity result for MMVs whose proof provides an explicit form of Theorem~\ref{thm:MMVparity}.
\begin{thm}
For positive integers $k_1\geq 1$ and $k_2,k_3\geq 2$, if weight $k_1+k_2+k_3$ is an even, then the triple $M$-values $M(k_1,k_2,k_3; \varepsilon_1,\varepsilon_2,\varepsilon_3)$ can be expressed as a rational linear combination of products of double $M$-values and single $M$-values.
\end{thm}
\begin{proof}  According to the definition of MMVs $M(k_1,k_2,\ldots,k_r;\varepsilon_1,\varepsilon_2,\ldots,\varepsilon_r)$, we can rewritten the triple $M$-values $M(k_1,k_2,k_3; \varepsilon_1,\varepsilon_2,\varepsilon_3)$  in the form
\begin{align*}
&M(k_1,k_2,k_3;\varepsilon_1,\varepsilon_2,\varepsilon_3)\\&=\sum_{0<m_1<m_2<m_3} \frac{(1+\varepsilon_1(-1)^{m_1})(1+\varepsilon_2(-1)^{m_2})(1+\varepsilon_3(-1)^{m_3})}{m_1^{k_1}m_2^{k_2}m_3^{k_3}}\\
&=\sum_{n=2}^\infty \frac{1+\varepsilon_2(-1)^n}{n^{k_2}} \sum_{i=1}^{n-1} \frac{1+\varepsilon_1(-1)^i}{i^{k_1}} \sum_{j=n+1}^\infty \frac{1+\varepsilon_3(-1)^j}{j^{k_3}}\\
&=M(k_3;\varepsilon_3)M(k_1,k_2;\varepsilon_1,\varepsilon_2)-\sum_{n=1}^\infty \frac{1+\varepsilon_2(-1)^n}{n^{k_2}} (H_{n-1}^{(k_1)}-\varepsilon_1 {\bar H}_{n-1}^{(k_1)})(H_{n}^{(k_3)}-\varepsilon_1 {\bar H}_{n}^{(k_3)}),
\end{align*}
where $H_n^{(k)}$ and ${\bar H}_n^{(k)}$ are denoted the harmonic number and alternating harmonic number, respectively, which for positive integers $k$ and $n$ are defined by
\begin{align*}
H^{(k)}_{n}:=\sum_{j=1}^n \frac{1}{j^k}, \quad {\bar H}^{(k)}_{n}:=\sum_{j=1}^n \frac{(-1)^{j-1}}{j^k}\quad \text{and}\quad H^{(k)}_{0}={\bar H}^{(k)}_{0}:=0.
\end{align*}
On the other hand, in \cite[Thm. 3.6 and Cor. 3.7]{X2020}, the first author proved the results that if
$p+q+m$ is even, and $q>1,m,p$ are positive integers, then
\begin{align*}
&\su \frac{H^{(m)}_nH^{(p)}_n}{n^q},\ \su \frac{{\bar H}^{(m)}_nH^{(p)}_n}{n^q},\ \su \frac{{\bar H}^{(m)}_n{\bar H}^{(p)}_n}{n^q}, \\
&\su \frac{H^{(m)}_nH^{(p)}_n}{n^q}(-1)^{n-1},\ \su \frac{{\bar H}^{(m)}_nH^{(p)}_n}{n^q}(-1)^{n-1},\ \su \frac{{\bar H}^{(m)}_n{\bar H}^{(p)}_n}{n^q}(-1)^{n-1}
\end{align*}
are reducible to (alternating) double zeta values.

Moreover, from the definition of alternating MZVs, it is easy to see that the (alternating) double zeta values can be expressed in terms of double $M$-values. Hence, we obtain the desired description.
\end{proof}

\section{Dimension Computation of MMVs}\label{sec:dimMMV}

Let $\MMV_w$ be the $\Q$-vector space generated by all the MMVs of weight $w$ and denote all its subspaces similarly. Set $\dim_\Q V_0=1$ for all the subspaces of $\MMV$ including $\MMV$ itself.
\begin{thm} \label{thm:MMVdim} Let $F_0=F_1=1$ and $F_n=F_{n-1}+F_{n-2}$ for all $n\ge 2$. Then
\begin{equation*}
 \dim_\Q \MMV_w\le F_w-1.
\end{equation*}
\end{thm}
\begin{proof}
Let $\ES_w$ be the $\Q$-vector space generated by the Euler sums of weight $w$.
Setting $\ES_0=1$, we have
the dimension bound $\ES_w\le F_w$ obtained by Deligne and Goncharov \cite{DeligneGo2005} and moreover a set of generators shown by Deligne \cite[Thm. 7.2]{Deligne2010}:
\begin{equation}\label{equ:DelBasis}
    \left\{ \prod \zeta(\bfk;1,\dots,1,-1) (2\pi i)^{2n}:
2n+\sum_\bfk \lambda(\bfk)|\bfk|=w, n\ge 0\right\},
\end{equation}
where the product runs through all possible Lyndon words $\bfk$
on odd numbers (with $1<3<5<\cdots$) with multiplicity $\lambda(\bfk)$ so that $2n+\sum_{\bfk} \lambda(\bfk)|\bfk|=w$.
Note that the ordering of indices in the definition of Euler sums is opposite in loc. sit. So the definition of Lyndon words here has opposite order, too.

Observe that the vector space freely generated by basis vectors corresponding to regularized Euler sums
of weight $w$ is dual to the degree $w$ part of the free associative algebra
$\Q\langle\!\langle e_0,e_{-1},e_1\rangle\!\rangle$. One expects that $\ES_w$
is a weighted polynomial algebra with $F_w$ generators in weight $w$. The reason is that
$\iota({\rm Lie} U_w)$ should be a Lie algebra freely generated by one element in each odd degree.

By the above consideration, one deduces that all Euler sums of the form
\begin{equation*}
\ES'=\Big\{\zeta(\bfk;\bfeps): \bfk\in\N^r, \bfeps\in\{\pm1\}^r, k_r>1, r\in\N\Big\},
\end{equation*}
can be generated by
\begin{equation*}
    \left\{ \prod \zeta(\bfk;1,\dots,1,-1) (2\pi i)^{2n}: 2n+\sum_\bfk \lambda(\bfk)|\bfk|=w,  n\ge 0,
\prod \zeta(\bfk;1,\dots,1,-1) \ne \zeta(-1)^w \right\}
\end{equation*}
which has cardinality $F_w-1$. We now show that $\ES'=\MMV$. The inclusion $\MMV\subseteq \ES'$ is obvious.
On the other hand, suppose $\bfk\in\N^r, \bfeps\in\{\pm1\}^r, k_r>1$. Then it is easy to see that
\begin{equation*}
\zeta(\bfk;\bfeps)=\sum_{0<n_1<\dots<n_r}\frac{\eps_1^{n_1} \dots \eps_r^{n_r}}{n_1^{k_1} \dots n_r^{k_r}}
=\sum_{\gd_1=0}^1 \cdots \sum_{\gd_r=0}^1
\sum_{\substack{0<n_1<\dots<n_r \\ n_j\equiv \gd_j \pmod{2} \ \forall j }}\frac{\eps_1^{n_1} \dots \eps_r^{n_r}}{n_1^{k_1} \dots n_r^{k_r}} \in \MMV.
\end{equation*}
This concludes the proof of the theorem.
\end{proof}

\begin{re}
(i). The basis \eqref{equ:DelBasis} has been used in the computation with Form in \cite{BlumleinBrVe2010}.
(ii). Notice that $\zeta(-1)^w=(-1)^w \log^w 2.$ The absence of these elements from MMVs is verified also by Maple computations using the basis \eqref{equ:DelBasis} adopted by \cite{BlumleinBrVe2010}.
\end{re}

\begin{cor} Let $\bfeps\in\{\pm1\}^r$ and $\bfk\in\N^r$ be an admissible composition of
positive integers of weight $w$. Then $\log (2)M(\bfk;\bfeps)$ can be expressed as a
$\Q$-linear combination of MMVs of weight $w+1$.
\end{cor}
\begin{proof}
This follows from the proof of Thm.~\ref{thm:MMVdim} since $\MMV_w$ is exactly the space of $\ES_w$ of codimension 1 with its complementary subspace generated by $\log^w 2$.
\end{proof}

 From numerical computation we have the following conjecture.
\begin{con}
We have the following generating functions
\begin{align*}
\sum_{w=0}^\infty  \dim_\Q \MtV_w t^w=&\, \frac{t}{1-t-t^2},\\
\sum_{w=0}^\infty  \dim_\Q \MMV_w t^w=\sum_{w=0}^\infty  \dim_\Q \MMVe_w t^w=&\, \frac{1}{1-t-t^2}-\frac{t}{1-t},
\end{align*}
and for all $k\in\N$
\begin{equation*}
 \dim_\Q \MTV_{2k}=\dim_\Q \MTV_{2k-1}+\dim_\Q \MTV_{2k-2}.
\end{equation*}
\end{con}

The conjectural dimensions for MtVs are already noticed by Hoffman in \cite{H2019} and he even suggested some possible basis for these subspaces of $\MMV$. We arrived at our conjectures from the data in Table~\ref{Table:dimMMV}.

\begin{table}[!h]
{
\begin{center}
\begin{tabular}{  |c|c|c|c|c|c|c|c|c|c|c|c|c|c|c| } \hline
       $w$     & 0 &  1  &  2  &  3  &  4  &  5  &  6  & 7  &  8  &  9  &  10 &  11  &  12  &  13  \\ \hline
$\dim_\Q \MtV_w$  & 1 &  0  &  1  &  2  &  3  &  5  &  8  & 13 &  21 &  34 &  55 &  89  &  144 &  233 \\ \hline
$\dim_\Q \MTV_w$  & 1 &  0  &  1  &  1  &  2  &  2  &  4  & 5  &  9  &  10 &  19 &  23  &  42  &  49 \\ \hline
$\dim_\Q \MSV_w$  & 1 &  0  &  1  &  2  &  3  &  4  &  6  & 10 &  15 &  22 &  32 &  52  &  76  &  ? \\ \hline
$\dim_\Q \MMV_w$  & 1 &  0  &  1  &  2  &  4  &  7  &  12 & 20 &  33 &  54 &  88 &  143 &  232 &  376 \\ \hline
$\dim_\Q \MMVe_w$ & 1 &  0  &  1  &  2  &  4  &  7  &  12 & 20 &  33 &  54 &  88 &  143 &  232 &  376 \\ \hline
$\dim_\Q \MMVo_w$ & 1 &  0  &  1  &  2  &  4  &  6  &  10 & 16 &  27 &  44 &  73 &  120 &  198 &  ? \\ \hline
\end{tabular}
\end{center}
}
\caption{Conjectural Dimensions of Various Subspaces of $\MMV$.}
\label{Table:dimMMV}
\end{table}

In summary, we have the following Venn diagram showing relations between all the above different variations of MZVs. Solid boundaries are all easily verified but the dashed boundaries mean the relations are still conjectural. So this diagram contains three conjectural relations:
\begin{equation*}
    \MSV\subseteq \MMVo,\quad  \MZV\subseteq \MtV\cap \MTV\cap \MSV, \quad \MMVe=\MMV.
\end{equation*}
where $\MZV$ is the $\Q$-vector space generated by MZVs.
Notice that the above conjectured relation in the middle implies \cite[Conj.~5.1]{KTA2019}
which says $\MZV\subseteq \MtV\cap \MTV$.

\begin{center}
\includegraphics[height=1.5in]{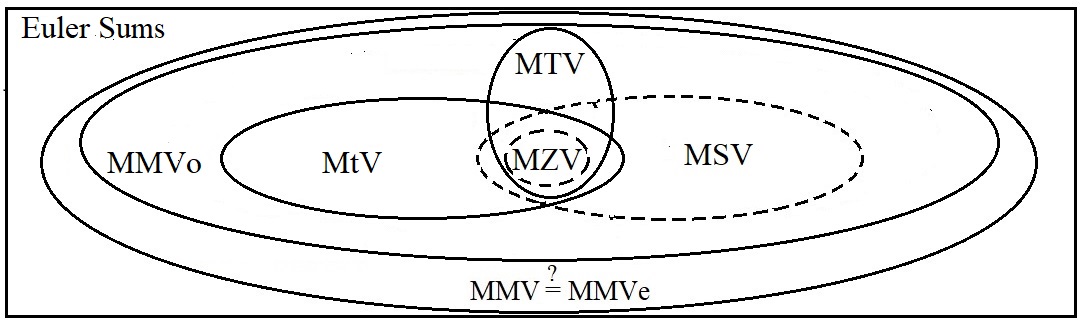}
\end{center}

\end{document}